\documentclass[a4paper,12pt]{article}

\usepackage[left=2cm,right=2cm, top=2cm,bottom=2cm,bindingoffset=0cm]{geometry}

\usepackage{verbatim}
\usepackage{amsmath}
\usepackage{amsthm}
\usepackage{amssymb}
\usepackage{delarray}
\usepackage{hyperref}
\usepackage{mathrsfs}

\newcommand{\al}{\alpha}

\newcommand{\de}{\delta}
\newcommand{\la}{\lambda}
\newcommand{\om}{\omega}

\newcommand{\eps}{\varepsilon}

\newcommand{\iy}{\infty}
\usepackage{graphicx}
\usepackage{tikz}
\usetikzlibrary{patterns}
\usepackage{caption}
\DeclareCaptionLabelSeparator{dot}{. }
\numberwithin{equation}{section}

\theoremstyle{plain}

\newtheorem{thm}{Theorem}[section]
\newtheorem{lem}{Lemma}[section]
\newtheorem{cor}{Corollary}[section]

\theoremstyle{definition}

\theoremstyle{remark}

\newtheorem{remark}{Remark}[section]

\DeclareMathOperator*{\Res}{Res}
\DeclareMathOperator{\rank}{rank}

 \allowdisplaybreaks

\begin{document}

\begin{center}
{\large\bf Inverse scattering on the line for the matrix Sturm-Liouville equation
}
\\[0.2cm]
{\bf Natalia Bondarenko} \\[0.2cm]
\end{center}

\vspace{0.5cm}

{\bf Abstract.} The inverse scattering problem is studied for the matrix Sturm-Liouville equation on the line.
Necessary and sufficient conditions for the scattering data are obtained.

\medskip


\medskip


\vspace{1cm}

{\large \bf 1. Introduction}

\setcounter{section}{1}
\setcounter{equation}{0}

\medskip

In this paper, we consider the matrix Sturm-Liouville (also called Schr\"odinger) equation on the real line:
\begin{equation} \label{eqv}
-Y'' + Q(x) Y = \rho^2 Y, \quad -\infty < x < \infty.
\end{equation}
Here $Y = [y_k(x)]_{k = 1}^m$ is a vector function, $\rho$ is the spectral parameter, and $Q(x) = [Q_{jk}(x)]_{j, k = 1}^n$
is the self-adjoint matrix potential $(Q(x) = Q^*(x))$, satisfying the condition
\begin{equation} \label{xQ}
\int_{-\infty}^{\infty} (1 + |x|) | Q_{jk}(x) | \, dx < \infty, \quad j, k = \overline{1, m}.
\end{equation} 
The inverse scattering problem is stidued, which consists in recovering of the potential $Q$ from the given scattering data.

There is an extensive literature on the inverse spectral and scattering problems 
(see monographs \cite{Mar77, Lev84, PT87, FY01, CS77} and references therein). Such problems arise
in quantum mechanics, geophysics, electronics, chemistry and other branches of science and engineering.
One of the most important applications is the inverse scattering method for integration of nonlinear evolution equations, 
such as Korteweg-de Vries (KdV) equation, nonlinear Schr\"odinger equation, sine-Gordon equation, Toda lattice, etc.
(see \cite{ZMNP, AS81, FT87, CD82}).

A complete analysis of the inverse scattering problem for the {\it scalar} Sturm-Liouville equation on the real line (equation \eqref{eqv} for $m = 1$)
was carried out by L.D. Faddeev \cite{Fad64} and V.A. Marchenko \cite{Mar77}. They have obtained the characterization of the scattering data,
in other words, necessary and sufficient contions for the solvability of the inverse scattering problem.

Inverse scattering problems for {\it matrix} Sturm-Liouville operators appeared to be more difficult, than scalar ones,
because of the more complicated structure of the discrete scattering data. 
For the matrix case, Z.S. Agranovich and V.A. Marchenko \cite{AM60} solved the inverse scattering problem of the {\it half-line},
using the trasformation operator method \cite{Mar77, Lev84}. Later on, several authors studied the inverse scattering problem 
for the matrix Sturm-Liouville operator on the {\it line} \cite{WK74, CD77, Wad80, AO82, Olm85, Schuur83, AG98}. 
In particular, M. Wadati and T. Kamijo \cite{WK74} reduced the inverse problem for the self-adjoint potential to the Gelfand-Levitan-Marchenko
equation (linear Fredholm integral equation, connecting the scattering data with the kernel of the transformation operator). 
E. Olmedilla \cite{Olm85} generalized their results to the non-self-adjoint case.
F. Calogero and A. Degasperis \cite{CD77} study applications of the inverse scattering problem to the integration of matrix nonlinear evolution equations
such as matrix KdV and Boomeron equations. However, as far as we know, a rigorous mathematical analysis of the solvability of the Gelfand-Levitan-Marchenko
equation and characterization of the scattering data were not done before for the matrix case. The goal of this paper is to cover this gap, and provide 
necessary and sufficient conditions for the solvability of the inverse scattering problem for equation \eqref{eqv}.

The paper is organized as follows. In Section~2, we introduce the left and the right scattering data for equation \eqref{eqv}, study their properties,
formulate the inverse scattering problem and give other preliminaries. In Section 3, the Gelfand-Levitan-Marchenko equation
is derived. Althouth many of the results of Sections 2 and 3 appeared before in \cite{AM60, WK74, AK01} and other literature,
we provide their proofs for the convenience of the reader. In Section 4, we give sufficient conditions for the unique solvability of the Gelfand-Levitan-Marchenko equation. 
However, these conditions are not sufficient for the solvability of the inverse scattering problem, because 
the solution of the main equation, constructed, for example, by the right scattering data, is guaranteed to fulfill \eqref{xQ} only on the right half-line.
Therefore we have to study the connection between the left and the right scattering data. 
In Section 5, we formulate and prove our main result, Theorem~\ref{thm:NSC} on the necessary and sufficient conditions
on the scattering data of the matrix Sturm-Liouville operator. We show how the left scattering data can be constructed
by the right ones. Then the resulting potential satisfies \eqref{xQ} on the full line, and we prove that its scattering data coincide with the given ones. 
We also investigate the case of reflectionless potentials 
with only discrete scattering data, and obtain the characterization for them as a corollary of the main theorem.
Finally, we discuss a connection of our necessary and sufficient conditions with the Riemann problem,
and apply them to the matrix KdV equation.

Let us introduce the notation.
We consider the Hilbert space of complex column $m$-vectors $\mathbb{C}^m$ with the following scalar product and the norm
$$
    (Y, Z) = \sum_{j = 1}^m \overline{y_j} z_j, \quad \| Y \| = \left( \sum_{j = 1}^m |y_k|^2 \right)^{1/2}, \quad Y = [y_j]_{j = \overline{1, m}}, \,
    Z = [z_j]_{j = \overline{1, m}},
$$
the space of row vectors $\mathbb{C}^{m, T}$, and the space of complex $m \times m$ matrices $\mathbb{C}^{m \times m}$ with the corresponding induced norm
$$
\| A \| = \max_{Y \in \mathbb{C}^m, \, \| Y \| = 1} \| A Y \|, \quad A = [a_{jk}]_{j, k = \overline{1, m}}.
$$
The symbols $I_m$ and $0_m$ are used for the unit $m \times m$ matrix and the zero $m \times m$ matrix, respectively.
The symbol ``$^*$'' denotes the conjugate transpose.

We use the notation $\mathcal{A}(\mathcal{I}; \mathbb{C}^m)$, $\mathcal{A}(\mathcal{I}; \mathbb C^{m, T})$ and
$\mathcal{A}(\mathcal{I}; \mathbb C^{m\times m})$ for classes of column vectors, row vectors and matrices, respectively,
with entries belonging to the class $\mathcal{A}(\mathcal{I})$
of scalar functions. The symbol $\mathcal{I}$ stands for an interval.
For example, the potential $Q$ belongs to the class $L((-\iy, \iy); \mathbb{C}^{m \times m})$. 

\bigskip

{\bf \large 2. Properties of the scattering data}

\medskip

\setcounter{section}{2}
\setcounter{equation}{0}

{\bf 2.1. Jost solutions.} 
Let us introduce matrix Jost solutions of equation \eqref{eqv} with the prescribed behavior 
at $\pm \infty$. Equation \eqref{eqv} has unique matrix solutions $F_{\pm}(x, \rho)$ for $\mbox{Im}\, \rho \ge 0$, 
such that $\lim\limits_{x \to \pm \infty} \exp(\mp i \rho x) F_{\pm}(x, \rho) = I_m$.
The Jost solutions $F_{\pm}(x, \rho)$ satisfy the integral equations
\begin{equation} \label{Jostint}
\begin{array}{l}
F_+(x, \rho) = \exp(i\rho x) \, I_m + \displaystyle\int_x^{\iy} \frac{\sin \rho (t - x)}{\rho} Q(t) F_+(t, \rho) \, dt, \\ [10pt]
F_-(x, \rho) = \exp(-i\rho x) \, I_m + \displaystyle\int_{-\iy}^x \frac{\sin \rho (x - t)}{\rho} Q(t) F_-(t, \rho) \, dt,
\end{array}
\end{equation}
and have the following {\bf properties}.

($i_1$) For each fixed $x$, the matrix functions $F^{(\nu)}_{\pm}(x, \rho)$, $\nu = 0, 1$, are
analytic for $\mbox{Im}\, \rho > 0$ and continuous for $\mbox{Im}\, \rho \ge 0$.

($i_2$) For $\nu = 0, 1$
\begin{equation} \label{Jostasymptx}
\begin{cases}
 F_+^{(\nu)}(x, \rho) = (i \rho)^{\nu} \exp(i \rho x)(I_m + o(1)), \quad x \to +\infty, \\
 F_-^{(\nu)}(x, \rho) = (-i \rho)^{\nu} \exp(-i\rho x)(I_m + o(1)), \quad x \to -\infty,
\end{cases}
\end{equation}  
uniformly in $\{ \rho \colon \mbox{Im}\,\rho \ge 0 \}$.

($i_3$) For each fixed $\rho$, $\mbox{Im}\,\rho \ge 0$, and real $a$, 
$$
 	F_+(x, \rho) \in L_2((a, \iy); \mathbb{C}^{m \times m}), \quad 
 	F_-(x, \rho) \in L_2((-\iy, a); \mathbb C^{m \times m}).
$$
Moreover, every vector solution of \eqref{eqv} from $L_2((a, \iy); \mathbb C^m)$ 
(or $L_2((-\iy, a); \mathbb C^{m}$) can be represented in the form
$F_+(x, \rho) V(\rho)$ (or $F_-(x, \rho) V(\rho)$, respectively), where $V(\rho) \in \mathbb C^m$.

($i_4$) For $|\rho|\to \iy$, $\mbox{Im}\,\rho \ge 0$, $\nu = 0, 1$,
\begin{equation} \label{Jostasymptrho}
\begin{cases}
 F_+^{(\nu)}(x, \rho) = (i \rho)^{\nu} \exp(i \rho x) \left( I_m + \dfrac{\om_+(x)}{i \rho} + o(\rho^{-1}) \right), 
 \quad \om_+(x) = -\dfrac{1}{2}\displaystyle\int_x^{\iy} Q(t) \, dt, \\
 F_-^{(\nu)}(x, \rho) = (-i \rho)^{\nu} \exp(- i \rho x) \left( I_m + \dfrac{\om_-(x)}{i \rho} + o(\rho^{-1}) \right), 
 \quad \om_-(x) = -\dfrac{1}{2} \displaystyle\int_{-\iy}^x Q(t)\, dt,
\end{cases}
\end{equation}
uniformly for $x \ge a$ and $x \le a$, respectively.

($i_5$) For real $\rho \ne 0$ the columns of the matrix functions $F_+(x, \rho)$ and $F_+(x, -\rho)$ (similarly, $F_-(x, \rho)$
and $F_-(x, -\rho)$) form a fundamental system of solutions for equation \eqref{eqv}.

($i_6$) There are trasformation operators for the Jost solutions:

\begin{equation} \label{transform}
\begin{cases}
F_+(x, \rho) = \exp(i \rho x) I_m + \displaystyle\int_x^{\iy} K_+(x, t) \exp(i \rho t) \, dt, \\
F_-(x, \rho) = \exp(-i \rho x) I_m + \displaystyle\int_{-\iy}^x K_-(x, t) \exp(-i\rho t)\, dt,
\end{cases}
\end{equation}
where the kernels $K_{\pm}(x, t)$ are continuous functions, having first derivatives with respect to $x$ and $t$,
the matrix functions
$$
   \frac{\partial}{\partial x} K_{\pm}(x, t) \pm \frac{1}{4} Q\left( \frac{x + t}{2}\right), \quad
   \frac{\partial}{\partial t} K_{\pm}(x, t) \pm \frac{1}{4} Q\left( \frac{x + t}{2}\right)
$$
are absolutely continuous with respect to $x$ and $t$,
and
\begin{equation} \label{recQ}
 	Q(x) = \mp 2 \frac{d}{dx} K_{\pm}(x, x).
\end{equation}
% other facts?

Note that, in fact, the Jost solutions are constructed for half-lines $(a, +\iy)$ and $(-\iy, a)$.
So one can find the proofs of the properties ($i_1$)-($i_6$), for example, in \cite{AM60}.

\medskip

{\bf 2.2. Matrix Wronskian.} Together with equation \eqref{eqv}, consider the following equation
\begin{equation} \label{eqv*}
-Z'' + Z Q(x) = \rho^2 Z, \quad -\iy < x < \iy,
\end{equation}
where $Z = Z(x)$ is a row vector. 
Define the matrix Wronskian $\langle Z, Y \rangle := Z'Y - Z Y'$.
If $Y(x, \la)$ and $Z(x, \la)$ satisfy equations \eqref{eqv} and \eqref{eqv*}, respectively, then
\begin{equation} \label{wron}
 	\frac{d}{dx} \langle Z(x, \la), Y(x, \la) \rangle = 0,
\end{equation}
so the expression $\langle Z(x, \la), Y(x, \la) \rangle$ does not depend on $x$.

One can introduce the Jost solutions $\bar F_{\pm}(x, \rho)$
for equation \eqref{eqv*} with properties, similar to ($i_1$)-($i_6$). Since the potential matrix $Q(x)$ is Hermitian,
one can easily check that 
\begin{equation} \label{barF}
\bar F_{\pm}(x, \rho) = F^*_{\pm}(x, -\bar \rho). 
\end{equation}
By virtue of \eqref{wron} and \eqref{Jostasymptx}
\begin{align} \label{wronF1}
   \langle \bar F_{\pm}(x, \rho),  F_{\pm}(x, \rho) \rangle & = \lim_{x \to \pm\iy} (\bar F_{\pm}'(x, \rho) F_{\pm}(x, \rho) - \bar F_{\pm}(x, \rho)  F'_{\pm}(x, \rho)) = 0_m, \\ \label{wronF2}
   \langle \bar F_{\pm}(x, \rho), F_{\pm}(x, -\rho) \rangle & = \lim_{x \to \pm\iy} (\bar F_{\pm}'(x, \rho) F_{\pm}(x, -\rho) - \bar F_{\pm}(x, \rho)  F'_{\pm}(x, -\rho)) = \pm 2 i \rho I_m,
\end{align}
for real $\rho \ne 0$.

\medskip

{\bf 2.3. Scattering matrix.} Since for real $\rho \ne 0$ the columns of the matrices $F_+(x, \rho)$ and $F_+(x, -\rho)$ 
(similarly, $F_-(x, \rho)$ and $F_-(x, -\rho)$) form a fundamental system of solutions for equation \eqref{eqv}, the following relations hold
\begin{align}\label{AB}
F_+(x, \rho) & = F_-(x, -\rho) A(\rho) + F_-(x, \rho) B(\rho), \\ \label{CD}
F_-(x, \rho) & = F_+(x, \rho) C(\rho) + F_+(x, -\rho) D(\rho),
\end{align}
with the $m \times m$ matrix coefficients $A(\rho)$, $B(\rho)$, $C(\rho)$, $D(\rho)$. Let us study the properties of these coefficients.

\begin{lem}
The following relations hold for real $\rho \ne 0$: 
\begin{gather} \label{ABrel}
B^*(-\rho) A(\rho) = A^*(-\rho) B(\rho), \quad
 A^*(\rho) A(\rho) = I_m + B^*(\rho) B(\rho), \\ \label{ABCDrel}
A(\rho) = D^*(-\rho), \quad B(\rho) = - C^*(\rho), \\ \label{wronA}
A(\rho) = -\frac{1}{2i\rho} \langle  \bar F_-(x, \rho), F_+(x, \rho) \rangle, \\ \label{wronB}
B(\rho) = \frac{1}{2i\rho} \langle  \bar F_-(x, -\rho), F_+(x, \rho) \rangle.
\end{gather}
\end{lem}

\begin{proof}
Using \eqref{AB} and \eqref{barF}, we derive
\begin{equation} \label{AB*}
\bar F_+(x, \rho) = A^*(-\rho) \bar F_-(x, -\rho) + B^*(-\rho) \bar F_-(x, \rho). 
\end{equation}
Consequently,
\begin{multline*}
\langle \bar F_+(x, \rho), F_+(x, \rho) \rangle = \langle A^*(-\rho) \bar F_-(x, -\rho) + B^*(-\rho) \bar F_-(x, \rho),  F_-(x, -\rho) A(\rho) + F_-(x, \rho) B(\rho) \rangle \\ =
- 2 i \rho B^*(-\rho) A(\rho) + 2 i \rho  A^*(-\rho) B(\rho) = 0_m.
\end{multline*}
Here we have applied \eqref{wronF1} and \eqref{wronF2}.
Similarly we obtain
\begin{multline*}
\langle \bar F_+(x, \rho), F_+(x, -\rho) \rangle = \langle A^*(-\rho) \bar F_-(x, -\rho) + B^*(-\rho) \bar F_-(x, \rho),  F_-(x, \rho) A(-\rho) \\ + F_-(x, -\rho) B(-\rho) \rangle  =
2 i \rho A^*(-\rho) A(-\rho) - 2i\rho B^*(-\rho) B(-\rho) = 2i I_m. 
\end{multline*}
Thus, the relations \eqref{ABrel} are proved.

The relations 
\begin{align*}
\langle \bar F_+(x, \rho), F_-(x, \rho) \rangle & =  \langle A^*(-\rho) \bar F_-(x, -\rho) + B^*(-\rho) \bar F_-(x, \rho), F_-(x, \rho) \rangle = 2 i \rho A^*(-\rho), \\
\langle \bar F_+(x, \rho), F_-(x, \rho) \rangle & = \langle \bar F_+(x, \rho), F_+(x, \rho) C(\rho) + F_+(x, -\rho) D(\rho) \rangle = 2 i \rho D(\rho), \\
\langle \bar F_+(x, \rho), F_-(x, -\rho) \rangle & =  \langle A^*(-\rho) \bar F_-(x, -\rho) + B^*(-\rho) \bar F_-(x, \rho), F_-(x, -\rho) \rangle = -2 i \rho B^*(-\rho), \\
\langle \bar F_+(x, \rho), F_-(x, -\rho) \rangle & = \langle \bar F_+(x, \rho), F_+(x, -\rho) C(-\rho) + F_+(x, \rho) D(-\rho) \rangle = 2 i \rho C(-\rho) 
\end{align*}
yield \eqref{ABCDrel}.

Using \eqref{AB}, \eqref{AB*}, \eqref{wronF1} and \eqref{wronF2} again, we obtain
\begin{align*}
\langle \bar F_-(x,\rho), F_+(x,\rho) \rangle & = \langle \bar F_-(x, \rho), F_-(x, -\rho) A(\rho) + F_-(x, \rho) B(\rho) \rangle = - 2 i \rho A(\rho), \\
\langle \bar F_-(x, -\rho), F_+(x, \rho) \rangle & = \langle \bar F_-(x, -\rho), F_-(x, -\rho) A(\rho) + F_-(x, \rho) B(\rho) \rangle = 2 i \rho B(\rho).
\end{align*}
Finally, we arrive at \eqref{wronA} and \eqref{wronB}.
\end{proof}	

The relation \eqref{wronA} gives the analytic continuation for $A(\rho)$ to the upper half-plane $\mbox{Im}\,\rho > 0$.
Hence, the matrix functions $A(\rho)$ and $D(\rho) = A^*(-\bar \rho)$ are analytic for $\mbox{Im}\,\rho > 0$ and 
$\rho A(\rho)$, $\rho D(\rho)$ are continuous for $\mbox{Im}\,\rho \ge 0$. The matrix functions $\rho B(\rho)$ and $\rho C(\rho)$
are continuous for real $\rho$.

\begin{lem} \label{lem:asymptABCD}
The following asymptotic formulas are valid:
\begin{gather*}
  	A(\rho), D(\rho) = I_m -\frac{\om}{i \rho} + o(\rho^{-1}), \quad \om = \frac{1}{2}\int_{-\iy}^{\iy} Q(t) \, dt \quad \mbox{Im}\,\rho \ge 0, \\
  	B(\rho), C(\rho) = o(\rho^{-1}), \quad \rho \in \mathbb{R},
\end{gather*}
as $|\rho| \to \iy$.
\end{lem}

\begin{proof}
The assertion of the lemma immediately follows from \eqref{ABCDrel}, \eqref{wronA}, \eqref{wronB} and \eqref{Jostasymptrho}.
\end{proof}

The matrix functions
\begin{equation} \label{defS}
  	S_-(\rho) = B(\rho) (A(\rho))^{-1}, \quad S_+(\rho) = C(\rho) (D(\rho))^{-1}
\end{equation}
are called the left and the right {\it scattering matrices}, respectively.
Denote
$$
 	F_+^0(x, \rho) = F_+(x, \rho) (A(\rho))^{-1}, \quad F_-^0(x, \rho) = F_-(x, \rho) (D(\rho))^{-1}. 
$$
It follows from \eqref{AB} and \eqref{CD}, that
$$
 	F_+^0(x, \rho) = F_-(x, -\rho) + F_-(x, \rho) S_-(\rho), \quad
 	F_-^0(x, \rho) = F_+(x, -\rho) + F_+(x, \rho) S_+(\rho).		
$$
Taking the asymptotic formulas \eqref{Jostasymptx} into account, we get
\begin{align*}
F_+^0(x, \rho) \sim \exp(i \rho x) I_m + S_-(\rho) \exp(-i\rho x), \quad & x \to -\iy, \\ 
F_+^0(x, \rho) \sim \exp(i\rho x) T_+(\rho), \quad & x \to +\iy, \\
F_-^0(x, \rho) \sim \exp(-i\rho x)I_m + S_+(\rho) \exp(i \rho x), \quad & x \to +\iy, \\
F_-^0(x, \rho) \sim \exp(-i \rho x) T_-(\rho), \quad & x \to -\iy.  	
\end{align*}
Thus, the matrix functions $S_{\pm}(\rho)$ generalize the scalar {\it reflection coefficients},
and the matrix functions $T_+(\rho) = (A(\rho))^{-1}$, $T_-(\rho) = (D(\rho))^{-1}$ 
generalize the {\it transmission coefficients} (see \cite[Chapter 3]{Mar77}).

\begin{lem} \label{lem:propS}
For real $\rho \ne 0$ the matrix functions $S_{\pm}(\rho)$ are continuous and the following relations hold:
\begin{gather} \nonumber
S^*_{\pm}(\rho) = S_{\pm}(-\rho), \quad \| S_{\pm}(\rho) \| < 1, \quad S_{\pm}(\rho) = o(\rho^{-1}), \quad |\rho| \to \iy, \\ \label{SA}
S_-^*(\rho) S_-(\rho) = I_m - (A^*(\rho))^{-1} (A(\rho))^{-1}, \quad
S_+^*(\rho) S_+(\rho) = I_m - (D^*(\rho))^{-1} (D(\rho))^{-1}, \\ \label{SAlim}
\lim_{\rho \to 0} \rho (S_-(\rho) + I_m) A(\rho) = 0_m, \quad \lim_{\rho \to 0} \rho (S_+(\rho) + I_m) D(\rho) = 0_m.
\end{gather}  
\end{lem}

\begin{proof}
In view of \eqref{ABrel}, $A^*(\rho) A(\rho) > 0$, therefore $\det A(\rho) \ne 0$, $\det D(\rho) \ne 0$
for real $\rho \ne 0$. Hence the scattering matrices $S_{\pm}(\rho)$, defined by \eqref{defS}, are continuous functions 
for real $\rho \ne 0$. By virtue of \eqref{defS} and \eqref{ABrel}, 
$$
  S_-^*(\rho) = (A^*(\rho))^{-1} B^*(\rho) = B(-\rho) (A(-\rho))^{-1} = S_-(-\rho),
$$
$$
 	I_m = (A^*(\rho))^{-1} (I_m + B^*(\rho) B(\rho)) (A(\rho))^{-1} = (A^*(\rho))^{-1} (A(\rho))^{-1} + S_-^*(\rho) S_-(\rho).
$$
Similar results are valid for $S_+(\rho)$. The estimate $\| S_{\pm}(\rho) \| < 1$ follows from \eqref{SA}.
Lemma~\ref{lem:asymptABCD} and \eqref{defS} yield $S_{\pm}(\rho) = o(\rho^{-1})$, as $|\rho| \to \iy$.

Using \eqref{AB}, \eqref{CD} and \eqref{defS}, we obtain
\begin{align*}
  	\rho F_+(x, \rho) & = (F_-(x, \rho) (S_-(\rho) + I_m) + F_-(x, -\rho) - F_-(x, \rho)) \rho A(\rho), \\
  	\rho F_-(x, \rho) & = (F_+(x, \rho) (S_+(\rho) + I_m) + F_+(x, -\rho) - F_+(x, -\rho)) \rho D(\rho).
\end{align*}
Taking limits as $\rho \to 0$, we arrive at \eqref{SAlim}.
\end{proof}

It follows from Lemma~\ref{lem:propS}, that the matrices $S_{\pm}(\rho)$ belong to $L_2((-\iy, \iy); \mathbb{C}^{m \times m})$,
so they have the Fourier transforms
\begin{equation} \label{defR}
 	R_{\pm}(x) = \frac{1}{2\pi} \int_{-\iy}^{\iy} S_{\pm}(\rho) \exp(\pm i \rho x)\, d \rho, 
\end{equation}
belonging to $L_2((-\iy, \iy); \mathbb{C}^{m \times m})$, $R_{\pm}(x) = R^*_{\pm}(x)$, and
\begin{equation} \label{inverseR}
 	S_{\pm}(\rho) = \int_{-\iy}^{\iy} R_{\pm}(x) \exp(\mp i \rho x) \, dx.
\end{equation}
 
\medskip

{\bf 2.4. Eigenvalues.} The values $\rho^2$, for which the equation \eqref{eqv} has nonzero solutions 
$Y(x) \in L_2((-\iy, \iy); \mathbb{C}^m)$, are called {\it eigenvalues} of \eqref{eqv}, and the
corresponding solutions are called {\it eigenfunctions}. 

\begin{lem} \label{lem:eigen1}
There are no eigenvalues for $\rho^2 \ge 0$.
\end{lem}

\begin{proof}
Let $\rho_0^2 > 0$ be an eigenvalue and $Y_0(x)$ be a corresponding eigenfunction. Then $Y_0(x) = F_+(x, \rho_0) A_0 + F_+(x, -\rho_0) B_0$,
where $A_0, B_0 \in \mathbb{C}^m$. But it follows from \eqref{Jostasymptx} and the relation $\lim\limits_{x \to +\iy} Y_0(x) = 0$, that
$A_0 = B_0 = 0$, so we arrive at the contradiction.

For $\rho = 0$ equation \eqref{eqv} has the solution $E(x) := F_+(x, 0) = I_m + o(1)$ as $x \to +\iy$. There exists
a constant $a$, such that $\det E(x) \ne 0$ for $x \ge a$. One can easily check that the matrix function
$$
 	Z(x) = E(x) \int_a^x (E^*(t) E(t))^{-1} dt
$$
satisfies equation \eqref{eqv} for $\rho = 0$ and enjoys asymptotic representation $Z(x) = x(I_m + o(1))$ as $x \to +\iy$.
Thus, the columns of the matrices $E(x)$ and $Z(x)$ form a fundamental system of solutions 
for equation \eqref{eqv} for $\rho = 0$. If equation \eqref{eqv} has the zero eigenvalue, then the corresponding 
eigenfunction admits the expansion $Y_0(x) = E(x) A_0 + Z(x) B_0$, $A_0, B_0 \in \mathbb C^m$. But in view of asymptotic behavior
of $Y_0(x)$, $E(x)$ and $Z(x)$, we obtain $A_0 = B_0 = 0$. Hence, $\rho^2 = 0$ is not an eigenvalue of \eqref{eqv}.
\end{proof}

Let 
$$ 
	\Lambda_+ := \{ \la = \rho^2 \colon \mbox{Im}\,\rho > 0, \, \det A(\rho) = 0 \}.
$$
Since $A(\rho)$ is analytic function in the upper half-plane and enjoys asymptotic representation $A(\rho) = I_m + O(\rho^{-1})$
as $|\rho| \to \iy$ by Lemma~\ref{lem:asymptABCD}, the set $\Lambda_+$ is bounded and at most countable.

\begin{lem} \label{lem:eigen2}
The set of eigenvalues coincide with $\Lambda_+$. If $A(\rho_0) V_0 = 0$, $V_0 \in \mathbb{C}^m$, $V_0 \ne 0$, then
$F_+(x, \rho_0) V_0$ is a vector eigenfunction, corresponding to the eigenvalue $\rho_0^2$.
\end{lem}

\begin{proof}
{\bf 1.} Let $\rho_0^2 \in \Lambda_+$. By virtue of \eqref{wronA}, there exists a vector $V_0 \in \mathbb C^m$, such that
$$
 	\langle \bar F_-(x, \rho_0), F_+(x, \rho_0) \rangle V_0 = 0.
$$
Then the vector function $Y_0(x) := F_+(x, \rho_0) V_0$, belonging to $L_2((a, +\iy); \mathbb C^m)$ for every real $a$,
satisfies the relation
$$
 	\langle \bar F_-(x, \rho_0), Y_0(x) \rangle = \bar F_-'(x, \rho_0) Y_0(x) - \bar F_-(x, \rho_0) Y'_0(x) = 0.
$$
Let $x_0$ be such that $\det F_-(x_0, \rho_0) \ne 0$ (such value exists in view of asymptotics \eqref{Jostasymptx}). Then
the solution $Y_0(x)$ satisfies $n$ linearly independent conditions:
$$
 	Y'_0(x_0) - (\bar F_-(x_0, \rho_0))^{-1} \bar F_-'(x_0, \rho_0) Y_0(x_0) = 0.
$$
The $n$ linearly independent columns of the matrix $F_-(x, \rho_0)$ satisfy the same conditions by virtue of \eqref{wronF1}.
Consequently, $Y_0(x) = F_-(x, \rho_0) U_0$, $U_0 \in \mathbb{C}^m$, so $Y_0(x) \in L_2((-\iy, a); \mathbb C^m)$.
Thus, $Y_0(x)$ is an eigenfuction, and $\rho_0^2$ is an eigenvalue of equation \eqref{eqv}.

{\bf 2.} On the contrary, let $\rho_0^2$ be an eigenvalue and $Y_0(x)$ be a corresponding eigenfunction. By virtue of
the property ($i_3$) of the Jost solutions,
$Y_0(x) = F_+(x, \rho_0) V_0 = F_-(x, \rho_0) U_0$, $V_0, U_0 \in \mathbb C^m$. On the one hand,
$$
 	\langle \bar F_-(x, \rho_0), Y_0(x) \rangle =  	\langle \bar F_-(x, \rho_0), F_-(x, \rho_0) \rangle U_0 = 0.	
$$
On the other hand,
$$
 	\langle \bar F_-(x, \rho_0), Y_0(x) \rangle =  	\langle \bar F_-(x, \rho_0), F_+(x, \rho_0) \rangle V_0 = 0.	
$$
In view of \eqref{wronA}, $\det A(\rho_0) = 0$, i.e. $\rho_0^2 \in \Lambda_+$.
\end{proof}

The operator $-Y'' + Q(x) Y$ is self-adjoint in $L_2((-\iy, \iy); \mathbb C^m)$. Taking Lemma~\ref{lem:eigen1} into account,
we conclude that eigenvalues $\rho^2$ of \eqref{eqv} are real and negative, and eigenfunctions, corresponding to different eigenvalues
$\la_k$ and $\la_n$ are orthogonal:
$$
 	\int_{-\iy}^{\iy} Y^*_k(x) Y_n(x) \, dx = 0.
$$
In view of \eqref{ABCDrel}, the eigenvalues also coincide with the zeros of $\det D(\rho)$ in the upper half-plane.

\begin{lem} \label{lem:finite}
The number of the eigenvalues is finite.
\end{lem}

\begin{proof}
Prove the assertion by contradiction. Suppose there is an infinite sequence $\{ \rho_k^2 \}_{k = 1}^{\iy}$ of negative eigenvalues,
and $\{ Y_k(x) \}_{k = 1}^{\iy}$ is an orthogonal sequence of corresponding vector eigenfunctions. Note that there can be multiple eigenvalues,
they occur in the sequence $\{ \rho_k^2 \}_{k = 1}^{\iy}$ multiple times with different eigenfunctions $Y_k(x)$.
The eigenfunctions can be represented in the form
$$
 	Y_k(x) = F_+(x, \rho_k) V_k = F_-(x, \rho_k) U_k, \quad \| V_k \| = \| U_k \| = 1,
$$	
Since $\rho_k = i \tau_k$, $\tau_k > 0$, and \eqref{barF} holds, we get
$$
  	Y_k^*(x) = V_k^* \bar F_+(x, \rho_k) = U_k^* \bar F_-(x, \rho_k).  	
$$

Using the orthogonality of the eigenfunctions, we obtain for $k \ne n$
\begin{multline} \label{ortYkn}
0 = \int_{-\iy}^{\iy} Y_k^*(x) Y_n(x) \, dx = V_k^* \int_a^{\iy} \bar F_+(x, \rho_k) F_+(x, \rho_n) \, dx \, V_n + 
U_k^* \int_{-\iy}^{-a} \bar F_-(x, \rho_k) F_-(x, \rho_n) \, dx \, U_n \\ + \int_{-a}^a Y_k^*(x) Y_k(x) \, dx + 
\int_{-a}^a Y_k^*(x) (Y_n(x) - Y_k(x)) \, dx =: \mathcal I_1 + \mathcal I_2 + \mathcal I_3 + \mathcal I_4.
\end{multline}
It follows from \eqref{Jostasymptx}, that 
$$
 	F_+(x, \rho_n) = \exp(-\tau_n x) (I_m + \al_n(x)), \quad \bar F_+(x, \rho_k) = \exp(-\tau_k x) (I_m + \al^*_k(x)),
$$
where $\| \al_k(x) \| \le \frac{1}{8}$ as $x \ge a$ for all $k \ge 1$ and for sufficiently large $a$. Therefore
\begin{multline*}
\mathcal{I}_1 = V_k^* \int_a^{\iy}  \exp(-(\tau_k + \tau_n)x)(I_m + \beta_{kn}(x)) \, dx \, V_k \\
+ V_k^* \int_a^{\iy} \exp(-(\tau_k + \tau_n)x)(I_m + \beta_{kn}(x)) \, dx (V_n - V_k). 
\end{multline*}	
Since the vectors $V_k$ belong to the unit sphere, one can choose a convergent subsequence 
$\{ V_{k_s} \}_{s = 1}^{\iy}$. Further we consider $V_k$ and $V_n$ from such subsequence.
Then for sufficiently large $k$ and $n$ we have
\begin{multline*}
\left| V_k^* \int_a^{\iy} \exp(-(\tau_k + \tau_n)x)(I_m + \beta_{kn}(x)) \, dx (V_n - V_k) \right| \\
\le \frac{3 \exp(-(\tau_k + \tau_n)a)}{2 (\tau_k + \tau_n)} \| V_n - V_k \| \le
\frac{\exp(-(\tau_k + \tau_n)a)}{8 (\tau_k + \tau_n)}.
\end{multline*}
Hence 
$$
\mathcal{I}_1 \ge \frac{\exp(-(\tau_k + \tau_n)a)}{2 (\tau_k + \tau_n)} \ge \frac{\exp(-2 a T)}{4 T}, \quad T := \max_k \tau_k.  	
$$
Similar estimate is valid for $\mathcal I_2$. Obviously, $\mathcal I_3 \ge 0$. Using the arguments, similar to the proof 
of \cite[Theorem 2.3.4]{FY01}, one can show that $\mathcal I_4 \to 0$ as $k, n \to \iy$. Thus, for sufficiently large 
$k$ and $n$, $\mathcal I_1 + \mathcal I_2 + \mathcal I_3 + \mathcal I_4 > 0$, that contradicts \eqref{ortYkn}.
Hence, the number of the eigenvalues is finite.
\end{proof}

Further we denote the set of eigenvalues by $\{ \rho_k^2 \}_{k = 1}^N$, $\rho_k = i \tau_k$, $\tau_k > 0$.

\medskip

{\bf 2.5 Scattering data}. 

\begin{lem} \label{lem:poles}
The poles of the matrix function $(A(\rho))^{-1}$ in the upper half-plane $\mbox{Im}\, \rho > 0$ 
are simple.
\end{lem}

We will prove Lemma~\ref{lem:poles}, using the following well-known fact (see \cite[Lemma 2.2.1]{AM60}).

\begin{lem} \label{lem:AM}
If there do not exist two nonzero vectors $a$ and $b$, such that
\begin{equation} \label{lemAM}
\begin{array}{c}
A(\rho) a=0, \\
\frac{d}{d\rho} A(\rho) a+ A(\rho) b=0,
\end{array}
\end{equation}
at some point $\rho_0$, then $\rho_0$ is a simple pole of $(A(\rho))^{-1}$.
\end{lem}

\begin{proof}[Proof of Lemma~\ref{lem:poles}]
Using \eqref{wronA}, calculate the derivative
\begin{multline*}
 	\frac{d}{d\rho} A(\rho) = -\frac{d}{d\rho} \left( \frac{1}{2 i \rho} \langle \bar F_-(x, \rho), F_+(x, \rho) \rangle \right) \\ =
 	-\frac{1}{2 i \rho} \left( -\frac{1}{\rho} \langle \bar F_-(x, \rho), F_+(x, \rho) \rangle +
 	\langle \tfrac{d}{d\rho} \bar F_-(x, \rho), F_+(x, \rho) \rangle +
 	\langle \bar F_-(x, \rho), \tfrac{d}{d\rho} F_+(x, \rho) \rangle \right).			 
\end{multline*}
Differentiate equation \eqref{eqv*} for $\bar F_-(x, \rho)$ by $\rho$:
\begin{equation} \label{eqvdrho}
-\frac{d}{d\rho} \bar F_-''(x, \rho) + \frac{d}{d\rho} \bar F_-(x, \rho) Q(x) = 2 \rho \bar F_-(x, \rho) + \rho^2 \frac{d}{d\rho} \bar F_-(x, \rho).	
\end{equation}
Using \eqref{eqv} and \eqref{eqvdrho}, one easily obtains
$$
  \frac{d}{dx} \langle \dfrac{d}{d\rho}\bar F_-(x, \rho), F_+(x, \rho) \rangle = -2 \rho \bar F_-(x, \rho) F_+(x, \rho). 
$$
Hence
$$
\langle \tfrac{d}{d\rho}\bar F_-(x, \rho), F_+(x, \rho) \rangle = \langle \tfrac{d}{d\rho}\bar F_-(x, \rho), F_+(x, \rho) \rangle_{x = -\iy}
- 2 \rho \int_{-\iy}^x \bar F_-(t, \rho) F_+(t, \rho)\, dt.  
$$
Similarly
$$
  \langle \bar F_-(x, \rho), \tfrac{d}{d\rho} F_+(x, \rho) \rangle = \langle \bar F_-(x, \rho), \tfrac{d}{d\rho} F_+(x, \rho) \rangle_{x = +\iy}
  + 2 \rho \int_x^{\iy} \bar F_-(t, \rho) F_+(t, \rho) \, dt. 
$$
Finally, we get
\begin{multline} \label{Aprime}
\frac{d}{d\rho} A(\rho) = -\frac{1}{\rho} A(\rho) - \frac{1}{2 i \rho} \Bigl(\langle \tfrac{d}{d\rho}\bar F_-(x, \rho), F_+(x, \rho) \rangle_{x = -\iy} \\ +
\langle \bar F_-(x, \rho), \tfrac{d}{d\rho} F_+(x, \rho) \rangle_{x = +\iy} + 2 \rho \int_{-\iy}^{\iy} \bar F_-(x, \rho) F_+(x, \rho) \, dx \Bigr). 
\end{multline}

By Lemma~\ref{lem:eigen2} the poles of $(A(\rho))^{-1}$ coincide with the eigenvalues of \eqref{eqv}.
Let $\rho_k$ be one of the poles. Then $A(\rho_k) V_k = 0$ if and only if $F_+(x, \rho_k) V_k$ is a vector eigenfunction,
corresponding to the eigenvalue $\rho_k^2$.
Note that 
\begin{equation} \label{sym1}
	F_+(x, \rho_k) V_k = F_-(x, \rho_k) U_k, \quad V_k^* \bar F_+(x, \rho_k) = U_k^* \bar F_-(x, \rho_k), \quad U_k \in \mathbb{C}^m.
\end{equation}

Using \eqref{Aprime}, \eqref{sym1} and \eqref{barF}, we derive
\begin{multline*}
U_k^* \frac{d}{d\rho} A(\rho_k) V_k = -\frac{1}{2i\rho_k} U_k^* \langle \tfrac{d}{d\rho}\bar F_-(x, \rho_k), F_-(x, \rho_k) \rangle_{x = -\iy} \, U_k 
\\ -\frac{1}{2i\rho_k} V_k^* \langle \bar F_+(x, \rho_k), \tfrac{d}{d\rho}F_+(x, \rho_k) \rangle_{x = +\iy} V_k +
i U_k^* \int_{-\iy}^{\iy} F_-^*(x, \rho_k) F_-(x, \rho_k) \, dx \, U_k. 
\end{multline*}
Similarly to the scalar case (see \cite[Theorem 3.4.1]{FY01}), using integral tramsforms \eqref{transform}, 
one can show that for $\nu = 0, 1$
\begin{equation} \label{asymptFderiv}
\left\{ 
    \begin{array}{l}
 	\dfrac{d}{d\rho} \bar F_-^{(\nu)}(x, \rho_k) = O(1), \quad x \to -\iy, \\ [10pt]
 	\dfrac{d}{d\rho} F_+^{(\nu)}(x, \rho_k) = O(1), \quad x \to +\iy.
 	\end{array}
\right.
\end{equation}
Applying these estimates together with \eqref{Jostasymptx}, we obtain
$$
\langle \tfrac{d}{d\rho}\bar F_-(x, \rho_k), F_-(x, \rho_k) \rangle_{x = -\iy} = 0_m, \quad
\langle \bar F_+(x, \rho_k), \tfrac{d}{d\rho}F_+(x, \rho_k) \rangle_{x = +\iy} = 0_m.
$$
Consequently,
$$
U_k^* \frac{d}{d\rho} A(\rho_k) V_k = i U_k^* \int_{-\iy}^{\iy} F_-^*(x, \rho_k) F_-(x, \rho_k) \, dx \, U_k \ne 0.
$$
Using \eqref{wronA}, \eqref{sym1} and \eqref{wronF1}, we derive
$$
U_k^* A(\rho_k) = -\frac{1}{2 i \rho_k} U_k^* \langle \bar F_-(x, \rho_k), F_+(x, \rho_k) \rangle =
-\frac{1}{2 i \rho_k} V_k^* \langle \bar F_+(x, \rho_k), F_+(x, \rho_k) \rangle = 0.
$$

Let $b \ne 0$ be some vector. Then
$$
 	U_k^* \frac{d}{d\rho} A(\rho_k) V_k + U_k^* A(\rho_k) b \ne 0,
$$
so \eqref{lemAM} can not hold for nonzero $a$ and $b$. By Lemma~\ref{lem:AM}, $\rho_k$
is a simple pole of $(A(\rho))^{-1}$.
\end{proof}

Denote
\begin{equation} \label{defRes}
 	R_k^- = \Res_{\rho = \rho_k} (A(\rho))^{-1} = \lim_{\rho \to \rho_k} (\rho - \rho_k) (A(\rho))^{-1}, \quad
 	R_k^+ = \Res_{\rho = \rho_k} (D(\rho))^{-1} = \lim_{\rho \to \rho_k} (\rho - \rho_k) (D(\rho))^{-1}.
\end{equation}
It follows from \eqref{defRes} and \eqref{ABCDrel}, that $R_k^- = -(R_k^+)^*$.

Clearly, 
$$
   A(\rho_k) R_k^- = \lim_{\rho \to \rho_k} (\rho - \rho_k) A(\rho) (A(\rho))^{-1} = 0_m.
$$
By Lemma~\ref{lem:eigen2}, $F_+(x, \rho_k) R_k^-$ is an eigenfunction, so it can be represented in the form
\begin{equation} \label{weight-}
F_+(x, \rho_k) R_k^- = i F_-(x, \rho_k) N_k^-.
\end{equation} 
Similarly
\begin{equation} \label{weight+}
F_-(x, \rho_k) R_k^+ = i F_+(x, \rho_k) N_k^+.
\end{equation}
We call the matrices $N_k^-$ and $N_k^+$ the left and the right {\it weight matrices}, respectively.

\begin{lem} \label{lem:Nk}
The weight matrices have the following properties:
\begin{equation} \label{rankNk}
 	\rank N_k^+ = \rank R_k^+ = \rank R_k^- = \rank N_k^-,
\end{equation}
\begin{equation} \label{Nk*}
 	N_k^{\pm} = (N_k^{\pm})^* \ge 0.
\end{equation}
\end{lem}

\begin{proof}
By virtue of \eqref{Jostasymptx}, $\det F_+(x, \rho_k) \ne 0$ for $x > a$ and $\det F_-(x, \rho_k) \ne 0$ for $x < -a$,
if $a$ is sufficiently large. Therefore it follows from \eqref{weight-} and \eqref{weight+}, that $\rank R_k^{\pm} = \rank N_k^{\pm}$.
The relation $R_k^- = -(R_k^+)^*$ implies $\rank R_k^+ = \rank R_k^-$.

Now let us prove \eqref{Nk*} for $N_k^+$. The case of $N_k^-$ is similar. One can easily show that
$$
   \frac{d}{dx} \langle \bar F_+(x, \rho_k), F_+(x, \rho) \rangle = (\rho^2 - \rho_k^2) \bar F_+(x, \rho_k) F_+(x, \rho).
$$ 
Hence
$$
 	\langle \bar F_+(t, \rho_k), F_+(t, \rho) \rangle\Big|_x^{\iy} = (\rho^2 - \rho_k^2) \int_x^{\iy} \bar F_+(t, \rho_k) F_+(t, \rho_k)\, dt.
$$
Using \eqref{Jostasymptx}, we obtain
\begin{equation} \label{smeq1}
 	\frac{1}{\rho^2 - \rho_k^2} i (N_k^+)^* \langle \bar F_+(x, \rho_k), F_+(x, \rho) \rangle i N_k^+ = (N_k^+)^* \int_x^{\iy} \bar F_+(x, \rho_k) F_+(x, \rho) \, dx \, N_k^+.
\end{equation}
It follows from \eqref{barF} and \eqref{weight+}, that
\begin{equation} \label{smeq2}
 	i (N_k^+)^* \bar F_+(x, \rho_k) = - (R_k^+)^* \bar F_-(x, \rho_k) = R_k^- \bar F_-(x, \rho_k).
\end{equation}
Using \eqref{smeq1}, \eqref{smeq2}, \eqref{wronA}, \eqref{defRes}, we derive
\begin{multline*}
\lim_{\rho \to \rho_k} \frac{1}{\rho^2 - \rho_k^2} i (N_k^+)^* \langle \bar F_+(x, \rho_k), F_+(x, \rho) \rangle i N_k^+ \\ =
\lim_{\rho \to \rho_k} \frac{1}{\rho^2 - \rho_k^2} (\rho - \rho_k) (A(\rho))^{-1} \langle \bar F_-(x, \rho), F_+(x, \rho) \rangle i N_k^+  \\ +
R_k^- \lim_{\rho \to \rho_k} \frac{1}{\rho^2 - \rho_k^2} \langle \bar F_-(x, \rho_k) - \bar F_-(x, \rho), F_+(x, \rho) \rangle i N_k^+ \\ =
\frac{1}{2 \rho_k} \lim_{\rho \to \rho_k} (A(\rho))^{-1} (-2 i \rho) A(\rho) i N_k^+ -
\frac{1}{2 \rho_k} R_k^- \langle \tfrac{d}{d\rho} \bar F_-(x, \rho), F_+(x, \rho) \rangle_{\rho = \rho_k} i N_k^+ \\ =
N_k^+ - \frac{1}{2 \rho_k} R_k^- \langle \tfrac{d}{d\rho} \bar F_-(x, \rho_k), F_-(x, \rho_k) \rangle R_k^+. 
\end{multline*}
By virtue of \eqref{Jostasymptx} and \eqref{asymptFderiv}, 
$$
 	\lim_{x \to -\iy} \langle \tfrac{d}{d\rho} \bar F_-(x, \rho_k), F_-(x, \rho_k) \rangle = 0_m.
$$
Passing to the limit as $\rho \to \rho_k$, $x \to -\iy$ in \eqref{smeq1}, we arrive at the relation
$$
 	N_k^+ = (N_k^+)^* \int_{-\iy}^{\iy} F_+^*(x, \rho_k) F_+(x, \rho_k) \, dx \, N_k^+.
$$
Hence \eqref{Nk*} is valid.
\end{proof}

\begin{remark}
The ranks of the weight matrices coincide with the multiplicities of the corresponding eigenvalues,
i.e. the number of corresponding linearly independent eigenfunctions. This fact can be proved similarly
to \cite[Lemma 4]{Bond11}.	
\end{remark}

The collections 
$$
J_+ := \{ \{ S_+(\rho) \}_{\rho \in \mathbb R}, \{ \rho_k^2, N_k^+ \}_{k = 1}^N \}, \quad
J_- := \{ \{ S_-(\rho) \}_{\rho \in \mathbb R}, \{ \rho_k^2, N_k^- \}_{k = 1}^N \}
$$
are called the {\it right} and the {\it left scattering data}, respectively.

\smallskip

{\bf Inverse scattering problem.} {\it Given the scattering data $J_+$ (or $J_-$), construct the potential matrix $Q$.}

The next lemma establishes the connection between the left and the right scattering data.

\begin{lem} \label{lem:connect}
The following relations hold
$$
S_-(\rho) = -D^*(\rho) S_+^*(\rho) (D^*(-\rho))^{-1}, \quad \rho \in \mathbb R \backslash \{ 0 \}, 	
$$
\begin{equation} \label{Nkconnect}
 	N_k^- = R_k^+ (N_k^+)^{-1} (R_k^+)^*, \quad k = \overline{1, N}.
\end{equation}
Since $N_k^+$ is not necessarily invertible, the matrix $(N_k^+)^{-1}$ is defined as follows:
$$
 	N_k^+ = U^* \mathcal D U, \quad U^* = U^{-1}, \, \mathcal D = \mbox{diag}\, \{ d_1, d_2, \dots, d_j, 0, \dots, 0 \}, \, d_l > 0, \, l = \overline{1, j}, 
$$
$$
 	(N_k^+)^{-1} = U^* \mathcal D^{-1} U, \quad \mathcal D^{-1} = \mbox{diag}\, \{ d_1^{-1}, d_2^{-1}, \dots, d_j^{-1}, 0, \dots, 0 \}.
$$
\end{lem}

\begin{proof}
Using \eqref{defS} and \eqref{ABCDrel}, we derive
\begin{multline*}
 	S_-(\rho) = B(\rho) (A(\rho))^{-1} = - C^*(\rho) (D^*(-\rho))^{-1} = (S_+(\rho) D(\rho))^* (D^*(-\rho))^{-1} \\ = -D^*(\rho) S_+^*(\rho) (D^*(-\rho))^{-1}.
\end{multline*}

In view of \eqref{Jostasymptx}, $\det F_-(x, \rho_k) \ne 0$ for $x < -a$, where $a$ is sufficiently large.
For such $x$, the relation \eqref{weight+} implies 
\begin{equation} \label{smeqRk}
 	R_k^+ = i C_k(x) N_k^+, \quad C_k(x) := (F_-(x, \rho_k))^{-1} F_+(x, \rho_k).
\end{equation}
Consequently,
\begin{equation} \label{smeqCk}
 	C_k(x) (N_k^+)^{1/2} = -i R_k^+ (N_k^+)^{-1/2},  
\end{equation}
where
$$
(N_k^+)^{-1/2} := U^* \mathcal D^{-1/2} U, \quad \mathcal D^{-1/2} = \mbox{diag}\, \{ d_1^{-1/2}, d_2^{-1/2}, \dots, d_j^{-1/2}, 0, \dots, 0 \},
$$
due to the notation above.
Using \eqref{weight-}, \eqref{smeqRk} and \eqref{Nk*}, we obtain
$$
 	N_k^- = -i C_k(x) R_k^- = i C_k(x) (R_k^+)^* = C_k(x) (N_k^+)^{1/2} ((N_k^+)^{1/2})^* C_k^*(x).
$$
Applying \eqref{smeqCk}, we arrive at \eqref{Nkconnect}.
\end{proof}

\medskip

{\bf 2.6. Behavior of $(A(\rho))^{-1}$ in the neighborhood of $\rho = 0$}

\begin{lem}
The folowing estimate is valid
\begin{equation} \label{estA0}
 (A(\rho))^{-1} = O(1), \quad \rho \to 0, \quad \mbox{Im}\,\rho \ge 0.
\end{equation}
\end{lem}

\begin{proof}
Introduce the potentials
$$
 	Q_r(x) = \begin{cases}
 	            Q(x), \quad |x| \le r, \\
 	            0, \quad |x| > r,
             \end{cases}
   \quad r > 0.
$$
Let $F_{\pm r}(x, \rho)$ be the corresponding Jost solutions.
According to the integral equations \eqref{Jostint}, the matrix functions $F_{\pm r}(x, \rho)$ 
are entire in $\rho$ for each fixed $x$, and
$$
 	\lim_{r \to \iy} \sup_{\mbox{Im}\,\rho \ge 0} \sup_{\pm x \ge a} \| (F_{\pm r}^{(\nu)}(x, \rho) - F_{\pm}^{(\nu)}(x, \rho) ) \exp(\mp i \rho x) \| = 0, 
 	\quad \nu = 0, 1, \, a \in \mathbb R.
$$
Denote
$$
 	A_r(\rho) := -\frac{1}{2 i \rho} \langle \bar F_{-r}(x, \rho), F_{+r}(x, \rho) \rangle.
$$
Obviously, the matrix function $\rho A_r(\rho)$ is entire in $\rho$ and
\begin{equation} \label{limAr}
\lim_{r \to \iy} \rho A_r(\rho) = \rho A(\rho)
\end{equation} 
uniformly for $\mbox{Im}\,\rho \ge 0$.

Let $( \rho_k^{(r)} )^2$, $k = \overline{1, N^{(r)}}$, be the eigenvalues of the potential $Q_r(x)$ (counted with their multiplicities).
Note that by virtue of Lemma~\ref{lem:asymptABCD} and 
\eqref{limAr}, 
\begin{equation} \label{eigenr}
T := \max_{k, r} |\rho_k^{(r)}| < \iy. 
\end{equation}
Let us show that $N^{(r)} \le N_*$. Assume to the contrary that $N^{(r_i)} \to \iy$ as $i \to \iy$.
Fix $\eps > 0$. For each $r$, choose the largest possible subset of eigenvalues, such that
$\| \rho_j^{(r)} - \rho_k^{(r)} \| < \eps$, $j, k = \overline{1, M^{(r)}}$ (without loss of generality, we may assume 
that the eigenvalues with indices $k = \overline{1, M^{(r)}}$ belong to this subset).In view of \eqref{eigenr}, we have
$M^{(r_i)} \to \iy$ as $i \to \iy$. The vector eigenfunctions, corresponding to the eigenvalues $\rho_k^{(r)}$, can be 
represented in the following form
$$
 	Y_{kr}(x) = F_{+r}(x, \rho_k^{(r)} ) V_k^{(r)} = F_{-r}(x, \rho_k^{(r)}) U_k^{(r)}, \quad V_k^{(r)}, \, U_k^{(r)} \in \mathbb C^m, \quad
 	\| V_k^{(r)} \| = \| U_k^{(r)} \| = 1.
$$
We assume that multiple eigenvalues are counted multiple times with pairwise orthogonal eigenfunctions.
Since the unit sphere in $R^m$ is compact, for sufficiently large $r$ in the sequence $\{ r_i \}$ one can choose a couple
of eigenvalues (call them $\rho_1^{(r)}$ and $\rho_2^{(r)}$), such that
\begin{equation} \label{smeqr}
 	\| \rho_1^{(r)} - \rho_2^{(r)} \| < \eps, \quad \| V_1^{(r)} - V_2^{(r)} \| < \eps.
\end{equation}

Following the proof of Lemma~\ref{lem:finite}, we use the orthogonality of the eigenfunctions:
\begin{multline} \label{ort2} 
0 = \int_{-\iy}^{\iy} Y_{1 r}^*(x) Y_{2 r}(x) \, dx = V_1^{(r)*} \int_a^{\iy} \bar F_{+r}(x, \rho_1^{(r)}) F_{+r}(x, \rho_2^{(r)}) \, dx \, V_2^{(r)} \\ + 
U_1^{(r)*} \int_{-\iy}^{-a} \bar F_{-r}(x, \rho_1^{(r)}) F_{-r}(x, \rho_2^{(r)}) \, dx \, U_2^{(r)} + \int_{-a}^a Y_{1r}^{*}(x) Y_{1r}(x) \, dx + 
\int_{-a}^a Y_{1r}^{*}(x) (Y_{2r}(x) - Y_{1 r}(x)) \, dx \\ =: \mathcal I_1 + \mathcal I_2 + \mathcal I_3 + \mathcal I_4.
\end{multline}
Similarly to the proof of Lemma~\ref{lem:finite}, one can show that
$$
\mathcal{I}_1, \, \mathcal{I}_2 \ge \frac{\exp(-2 a T)}{4 T}, \quad \mathcal I_3 \ge 0,
$$
where $T$ was defined in \eqref{eigenr}. It follows from \eqref{smeqr}, that $\| \mathcal{I}_4 \| \le C \eps$, where the constant $C$
does not depend on $r$. Choosing sufficiently small $\eps$, we obtain $\mathcal{I}_1 + \mathcal{I}_2 + \mathcal{I}_3 + \mathcal{I}_4 > 0$
and arrive at a contradiction with \eqref{ort2}. Thus, $N^{(r)} \le N_*$.

Denote
$$
  \tau^* := \frac{1}{2} \min \tau_k, \quad \mathcal{D} := \{ \rho \colon |\rho| < \tau^*, \, \mbox{Im}\, \rho > 0 \}, \quad
  \mathcal{D}_{\delta} := \{ \rho \colon \mbox{Im}\,\rho > 0, \, \de < |\rho| < \tau^* \},
$$
$$
  P_r(\rho) := \prod_{\rho_k^{(r)} \in \mathcal{D}} \frac{\rho - \rho_k^{(r)}}{\rho + \rho_k^{(r)}}, \quad H_r(\rho) := P_r(\rho) (A_r(\rho))^{-1}.
$$
By virtue of \eqref{limAr}, the set $\{ \rho_k^{(r)} \}$ does not have limit points in $\overline{ \mathcal{D}}$ other than $\rho = 0$.
Note that $\| (A_r(\rho))^{-1} \| \le 1$ holds for real $\rho \ne 0$, since \eqref{SA} holds for $A_r(\rho)$. Therefore, $(A_r(\rho))^{-1}$ is regular at $\rho = 0$,
so it is bounded in $\partial \mathcal D$ for sufficiently large $r$. Since $N^{(r)} \le N_*$ and $\rho_k^{(r)}$ in the product $P_r$ tend to zero, 
$\| P_r(\rho) \| \le C$ in $\partial \mathcal D$. Consequently, for sufficiently large $r$, $\| H_r(\rho) \| \le C$ in $\partial \mathcal D$,
where the constant $C$ does not depend on $r$. The matrix function $H_r(\rho)$ is analytic in $\mathcal D$ and continuous in 
$\overline{ \mathcal D}$, therefore the estimate $\| H_r(\rho) \| \le C$ holds in $\overline{ \mathcal D}$.

It follows from \eqref{limAr}, that
$$
\lim_{r \to \iy} (A_r(\rho))^{-1} = (A(\rho))^{-1}, 
$$
uniformly by $\rho \in \overline{ \mathcal D}_{\de}$. Clearly, $P_r(\rho) \to 1$ as $r \to \iy$. Hence
$\lim\limits_{r \to \iy} H_r(\rho) = (A(\rho))^{-1}$ uniformly in $\overline{\mathcal D}_{\de}$. Consequently,
$\| (A(\rho))^{-1} \| \le C$ in $\overline{\mathcal D_{\de}}$. Since $C$ does not depend on $\de$, we arrive at~\eqref{estA0}.
\end{proof}

\begin{remark}
A more detailed analysis of asymptotic behavior of the scattering matrix near $\rho = 0$ is presented in \cite{AK01}.
\end{remark}

\medskip

{\large \bf 3. Derivation of the Gelfand-Levitan-Marchenko equation}

\setcounter{section}{3}
\setcounter{equation}{0}

\medskip

In this section, we show the reduction of the studied inverse scattering problem to the linear integral equation.
Although Gelfand-Levitan-Marchenko equation was used before (see, for example \cite{WK74, Olm85}),
we provide its derivation in order to make the paper self-contained.

\begin{thm} \label{thm:GLM}
For each fixed $x$, the matrix functions $K_{\pm}(x, t)$ (see \eqref{transform}) satisfy the
Gelfand-Levitan-Marchenko equations
\begin{equation} \label{GLM+}
M_+(x + y) + K_+(x, y) + \int_x^{\iy} K_+(x, t) M_+(t + y) \, dt = 0_m, \quad y > x,
\end{equation}
\begin{equation} \label{GLM-}
M_-(x + y) + K_-(x, y) + \int_{-\iy}^x K_-(x, t) M_-(t + y) \, dt = 0_m, \quad y < x,
\end{equation}
where
\begin{equation} \label{defM}
M_{\pm}(x) = R_{\pm}(x) + \sum_{k = 1}^N N_k^{\pm} \exp(\mp \tau_k x),
\end{equation}
the matrix functions $R_{\pm}(x)$ are defined in \eqref{defR}, and $\rho_k = i \tau_k$.
\end{thm}

\begin{proof}
The relation \eqref{CD} can be rewritten in the form
\begin{equation} \label{smeq3}
  F_-(x, \rho) ((D(\rho))^{-1} - I_m) = F_+(x, \rho) S_+(\rho) + F_+(x, -\rho) - F_-(x, \rho).
\end{equation}
Put $K_+(x, t) = 0_m$ for $t < x$ and $K_-(x, t) = 0_m$ for $t > x$.
By virtue of \eqref{transform} and \eqref{defR}
\begin{multline*}
F_+(x, \rho) S_+(\rho) + F_+(x, -\rho) - F_-(x, \rho) = \Bigl( \exp(i \rho x) + \int_{-\iy}^{\iy} K_+(x, t) \exp(i \rho t) \, dt \Bigr) \\
\cdot \Bigl( \int_{-\iy}^{\iy} R_+(y) \exp(i \rho y) \, dy \Bigr)  + \int_{-\iy}^{\iy} (K_+(x, t) - K_-(x, t)) \exp(-i\rho t) \, dt \\ =: 
\int_{-\iy}^{\iy} H(x, y) \exp(-i\rho y) \, dy.
\end{multline*}
where
$$
 H(x, y) = K_+(x, y) - K_-(x, y) + R_+(x + y) + \int_x^{\iy} K_+(x, t) R_+(t + y) \, dt.
$$
Thus, for each fixed $x$, the right-hand side of \eqref{smeq3} is the Fourier transform of $H(x, y)$.
Applying the inverse Fourier transform, we get
\begin{equation} \label{defH}
  H(x, y) = \frac{1}{2\pi} \int_{-\iy}^{\iy} F(\rho) \, d\rho, \quad F(\rho) :=  F_-(x, \rho) ((D(\rho))^{-1} - I_m) \exp(i \rho y).
\end{equation}
Using \eqref{Jostasymptrho} and Lemma \ref{lem:asymptABCD}, we obtain the following asymptotic formula for fixed $x < y$:
\begin{equation} \label{smasymptF}
 	F(\rho) = \frac{1}{i \rho} \exp(i \rho(y - x)) (\om + o(1)), \quad |\rho| \to \iy, \, \mbox{Im}\, \rho \ge 0.
\end{equation}
The matrix function $F(\rho)$ is meromorphic in the upper-half plane with the poles $\rho_k$, $k = \overline{1, N}$.
Let $C_{\de, R}$ be a closed contour (with the counterclockwise circuit), which is the boundary of the domain
$D_{\de, R} := \{ \rho \colon \mbox{Im}\,\rho > 0, \, \de < |\rho| < R \}$. The residue theorem yields
$$
 	\frac{1}{2 \pi i} \int_{C_{\de, R}} F(\rho) \, d\rho = \sum_{k = 1}^N \Res_{\rho = \rho_k} F(\rho).
$$
It follows from the estimates \eqref{Jostasymptrho}, \eqref{estA0} and \eqref{smasymptF}, that
$$
 	\lim_{R \to \iy, \, \eps \to 0} \frac{1}{2 \pi i} \int_{C_{\de, R}} F(\rho) \, d\rho = \frac{1}{2 \pi i} \int_{-\iy}^{\iy} F(\rho)\,d\rho.
$$
Hence 
\begin{equation} \label{sumH}
   H(x, y) = i \sum_{k = 1}^N \Res_{\rho = \rho_k} F(\rho).
\end{equation}
Using \eqref{weight+} and \eqref{transform}, we obtain 
\begin{multline} \label{calcResF}
 	\Res_{\rho = \rho_k} F(\rho) = F_-(x, \rho_k) R_k^+ \exp(i \rho_k y) = F_+(x, \rho_k) N_k^+ \exp(-\tau_k x) \\ =
 	N_k^+ \exp(-\tau_k (x + y)) + \int_x^{\iy} K_+(x, t) N_k^+ \exp(-\tau_k (t + y)) \, dt.
\end{multline}
Thus, combining \eqref{defH}, \eqref{sumH} and \eqref{calcResF}, and taking \eqref{defM} into account, we arrive at equation \eqref{GLM+}. Equation \eqref{GLM-}
can be derived similarly.         
\end{proof}

\begin{cor} \label{cor:M}
The matrix functions $M_{\pm}(x)$ are absolutely continuous, and for each fixed $a > -\iy$ the following
estimates are valid:
\begin{equation*}
\int_a^{\iy} \| M_{\pm}(\pm x) \| \, dx < \iy, \quad \int_a^{\iy} (1 + |x|) \| M'_{\pm}(\pm x) \| \, dx < \iy.
\end{equation*}
\end{cor}

\begin{proof}
The proof is quite similar to the scalar case (see \cite[Chapter 3]{Mar77}, \cite[Lemma 3.2.2]{FY01}).
\end{proof}

\medskip

{\large \bf 4. The unique solvability of the Gelfand-Levitan-Marchenko equation}

\setcounter{section}{4}
\setcounter{equation}{0}

\medskip

The unique solvability of equations \eqref{GLM+} and \eqref{GLM-} plays a key role in the analysis
of the inverse scattering problem.

Let us introduce the following conditions (Condition A$_+$ for the data $J_+$ and Condition A$_-$ for the data $J_-$).

\smallskip

{\bf Condition A$_\pm$.} For real $\rho \ne 0$, the matrix functions $S_{\pm}(\rho)$ are continuous, and
$$
 	\| S_{\pm}(\rho) \| < 1, \quad S_{\pm}^*(\rho) = S_{\pm}(-\rho), \quad S_{\pm}(\rho) = o(\rho^{-1}), \quad |\rho| \to \iy. 
$$
The matrix functions 
$$
 	R_{\pm}(x) = \int_{-\iy}^{\iy} S_{\pm}(\rho) \exp(\pm i \rho x) \, d\rho
$$ 
are absolutely continuous, $R_{\pm} \in L_2((-\iy, \iy);\mathbb C^{m \times m})$, $R_{\pm}(x) = R_{\pm}^*(x)$, and for
each fixed $a > -\iy$ the following estimates hold
\begin{equation} \label{estR}
\int_a^{\iy} \| R_{\pm}(\pm x) \| \, dx < \iy, \quad \int_a^{\iy} (1 + |x|) \| R'_{\pm}(\pm x) \| \, dx < \iy.
\end{equation}
Moreover, $\rho_k = i \tau_k$, $\tau_k > 0$, $N_k^{\pm} = (N_k^{\pm})^* \ge 0$, $k = \overline{1, N}$.

\begin{thm} \label{thm:GLMsol}
Let the collection $J_+ = \{ S_+(\rho), \rho_k^2, N_k^+ \}$ (or $J_-$) satisfies Condition A$_+$ (A$_-$).
Then for each fixed $x$, equation \eqref{GLM+} (respectively, \eqref{GLM-}) has a unique solution 
$K_+(x, y) \in L((x, \iy); \mathbb C^{m \times m})$ (respectively, $K_-(x, y) \in L((-\iy, x); \mathbb{C}^{m \times m})$).
\end{thm}

\begin{proof}
One can easily show that for each fixed $x$, the operator
$$
 	(J_x f)(y) = \int_x^{\iy} f(t) M_+(t + y) \, dt, \quad y > x,
$$
is compact in the space of row-vector functions $L((x, \iy); \mathbb{C}^{m, T})$.
So it is sufficient to prove that the homogeneous equation
\begin{equation} \label{homo}
 	f(y) + \int_x^{\iy} f(t) M_+(t + y) \, dt = 0
\end{equation}
has only zero solution. By virtue of \eqref{defM}, \eqref{estR} and \eqref{homo}, the vector function $f(y)$ is bounded for $y > x$.
Hence $f \in L_2((x, \iy); \mathbb C^{m, T})$. So we derive from \eqref{homo}:
$$
 	\int_x^{\iy} f(y) f^*(y) \, dy + \int_x^{\iy} \int_x^{\iy} f(t) M_+(t + y) f^*(y) \, dt \, dy = 0.
$$
Using \eqref{defM}, we obtain
\begin{multline} \label{smeq4}
 	\int_x^{\iy} f(y) f^*(y) \, dy + \sum_{k = 1}^N \int_x^{\iy} \int_x^{\iy} f(t) \exp(-\tau_k t) N_k^+ \exp(-\tau_k y) f^*(y) \, dt \, dy \\ +
 	\frac{1}{2 \pi} \int_{-\iy}^{\iy} \int_x^{\iy} \int_x^{\iy} f(t) \exp(i \rho t) S_+(\rho) \exp(i \rho y) f^*(y) \, dt \, dy \, d \rho = 0.
\end{multline}
Denote 
$$
 	\Phi(\rho) = \int_x^{\iy} f(t) \exp(i \rho t) \, dt.
$$
By Parseval's identity,
$$
 	\int_x^{\iy} f(y) f^*(y) \, dy = \frac{1}{2\pi} \int_{-\iy}^{\iy} \Phi(\rho) \Phi^*(\rho) \, d\rho.
$$
Then the relation \eqref{smeq4} takes the form
\begin{equation} \label{smeq5}
 	\frac{1}{2\pi} \int_{-\iy}^{\iy} \Phi(\rho) (I_m + S(\rho)) \Phi^*(\rho) \, d \rho + \sum_{k = 1}^N \Phi(\rho_k) N_k^+ (\Phi(\rho_k))^* = 0.
\end{equation}
Let 
$$
 	H(\rho) = \frac{1}{2} (S_+(\rho) + S_+^*(\rho)).
$$
Then
$$
 	\mbox{Re}\, \left\{ \int_{-\iy}^{\iy} \Phi(\rho) (I_m + S_+(\rho)) \Phi^*(\rho) \, d \rho \right\} =
 	\int_{-\iy}^{\iy} \Phi(\rho) (I_m + H(\rho)) \Phi^*(\rho) \, d\rho. 
$$
It follows from $\| S_+(\rho) \| < 1$, that $\| H(\rho) \| < 1$, and, consequently, 
$I_m + H(\rho) > 0$ for real $\rho \ne 0$. Hence
$$
 	\int_{-\iy}^{\iy} \Phi(\rho) (I_m + H(\rho)) \Phi^*(\rho) \, d\rho \ge 0.
$$
Note that
$$
  \Phi(\rho_k) N_k^+ (\Phi(\rho_k))^* \ge 0.	
$$
Therefore, the equality \eqref{smeq5} is possible only if 
$$
 	\int_{-\iy}^{\iy} \Phi(\rho) (I_m + H(\rho)) \Phi^*(\rho) \, d\rho = 0.
$$
This relation yields $\Phi(\rho) \equiv 0$, so $f(y) \equiv 0$. Thus, the homogeneous equation \eqref{homo} has a unique solution,
and the unique solvability of \eqref{GLM+} is proved. The proof for \eqref{GLM-} is similar.
\end{proof}

Note that the data $J_+$ and $J_-$ in Theorem~\ref{thm:GLMsol} are not required to be 
scattering data for some particular potential $Q(x)$. However, by virtue of Lemmas~\ref{lem:propS}, \ref{lem:finite}, \ref{lem:Nk},
Corollary~\ref{cor:M}
and other results of Section~2, Conditions A$_\pm$ hold for the scattering data $J_{\pm}$ of equation \eqref{eqv}.

\begin{cor}[uniqueness theorem]  \label{cor:uniq}
The potential $Q$ in equation \eqref{eqv} is uniquely determined by the scattering data $J_+$ (or $J_-$).
\end{cor}

The following algorithm can be used for the solution of the inverse scattering problem.

\smallskip

{\bf Algorithm.} Given the scattering data $J_+ = \{ S_+(\rho), \rho_k^2, N_k^+ \}$.
\begin{enumerate}
\item Construct $M_+(x)$ by formula \eqref{defM}.
\item Solve equation \eqref{GLM+} and find $K_+(x, y)$.
\item Find the potential $Q(x)$ by \eqref{recQ}.
\end{enumerate}

The solution for $J_-$ is similar.

Further we also need the following auxiliary fact.

\begin{lem} \label{lem:Qpm}
Let the collections $J_{\pm}$ satisfy Comditions A$_\pm$, and let $K_{\pm}(x, y)$ be the solutions of the
integral equations \eqref{GLM+} and \eqref{GLM-}. Define the matrix functions $F_{\pm}(x, \rho)$ by \eqref{transform}, 
and the matrix functions $Q_{\pm}(x)$ by \eqref{recQ}. Then for each fixed $a > -\iy$
$$
 	\int_a^{\iy} (1 + |x|) \| Q_{\pm}(\pm x) \| \, dx < \iy,
$$
and
$$
 	-F''_{\pm}(x, \rho) + Q_{\pm}(x) F_{\pm}(x, \rho) = \rho^2 F_{\pm}(x, \rho).
$$
\end{lem}

\begin{proof}
The proof is quite technical and repeats the proof of \cite[Lemma 3.3.1]{FY01}.
\end{proof}

\medskip
{\large \bf 5. Necessary and sufficient conditions}

\setcounter{section}{5}
\setcounter{equation}{0}

\medskip

{\bf 5.1. Main theorem.}
In this section, we formulate and prove the main result on the necessary and sufficient conditions
for the solvability of the inverse scattering problem.

Note that Lemma~\ref{lem:Qpm}, as in the scalar case, gives a ``good'' behavior of the potential on $(a, \iy)$,
if we solve the Gelfand-Levitan-Marchenko equation by $J_+$, and on $(-\iy, a)$, if we use $J_-$. Therefore we need 
the connection between the left and the right scattering data in the necessary and sufficient conditions.
For definiteness, we formulate the following condition for $J_+ = \{ S_+(\rho), \rho_k^2, N_k^+ \}$.

\medskip

{\bf Condition B.} There exists a matrix function $D(\rho)$, such that:

\begin{enumerate}
\item $D(\rho)$ is analytical for $\mbox{Im}\,\rho > 0$ and $\rho D(\rho)$ is continuous for $\mbox{Im}\,\rho \ge 0$.

\item $\det D(\rho) = 0$ only for $\rho = \rho_k$, $k = \overline{1, N}$, and $\Res\limits_{\rho = \rho_k} (D(\rho))^{-1} = C_k N_k^+$
for some $C_k \in \mathbb{C}^{m \times m}$, $\det C_k \ne 0$.

\item $D(\rho) = I_m + O(\rho^{-1})$, as $|\rho| \to \iy$.

\item $(D(\rho))^{-1} = O(1)$, as $\rho \to 0$.

\item $(D^*(\rho))^{-1} (D(\rho))^{-1} = I_m - S^*_+(\rho) S_+(\rho)$ for real $\rho \ne 0$.

\item $\lim\limits_{\rho \to 0} \rho (S_+(\rho) + I_m) D(\rho) = 0_m$, $\rho \in \mathbb R$.

\item Define 
\begin{equation} \label{recS-}
S_-(\rho) = -D^*(\rho) S_+^*(\rho) (D^*(-\rho))^{-1}
\end{equation} 
and $R_-(x)$ by formula \eqref{defR}. The the matrix function
$R_-(x)$ is absolutely continuous and for each fixed $a > -\iy$
$$
   \int_{-\iy}^a \| R_-(x) \| \, dx < \iy, \quad \int_{-\iy}^a (1 + |x|) \| R'_-(x) \| \, dx < \iy.
$$

\end{enumerate}

\begin{remark}
In the scalar case ($m = 1$) the function $D(\rho)$, satisfying Conditions B1--B5, is uniquely determined by 
the following formula (\cite[Theorem 3.5.1]{Mar77}, \cite[Theorem 3.3.1]{FY01}):
$$
 	D(\rho) = \prod_{k = 1}^N \frac{\rho - i \tau_k}{\rho + i \tau_k} \exp(\gamma(\rho)), \quad
 	\gamma(\rho) = -\frac{1}{2 \pi i} \int_{-\iy}^{\iy} \frac{\ln(1 - |S_+(\xi)|^2)}{\xi - \rho} \, d\xi, \quad \mbox{Im}\,\rho > 0. 
$$
\end{remark}

\begin{thm}[Necessary and sufficient conditions] \label{thm:NSC}
For data $J_+ = \{ \{ S_+(\rho)\}_{\rho \in \mathbb R}, \{ \rho_k^2, N_k^+ \}_{k = 1}^N \}$ to be the right scattering data
for a certain potential $Q = Q^*$, satisfying \eqref{xQ}, it is necessary and sufficient to satisfy Condition A$_+$ and 
Condition B.
\end{thm}

The necessity part of Theorem~\ref{thm:NSC} was proved in Sections~2 and 4. Let us prove the sufficiency part.
Let the data $J_+$ satisfy Condition A$_+$ and Condition B. Construct $A(\rho) := D^*(-\bar \rho)$, 
$R_k^+ = \Res\limits_{\rho = \rho_k} (D(\rho))^{-1}$, $R_k^- := \Res\limits_{\rho = \rho_k} (A(\rho))^{-1} = -(R_k^+)^*$  
and $N_k^-$ by formula \eqref{Nkconnect}. 

\begin{lem}
The data $J_- = \{ S_-(\rho), \rho_k^2, N_k^- \}$, constructed above, satisfy Condition A$_-$. 
\end{lem}

\begin{proof}
It follows from \eqref{recS-}, Conditions B1, B3 and A$_+$, that the matrix function
$S_-(\rho)$ is continuous for real $\rho \ne 0$ and $S_-(\rho) = o(\rho^{-1})$ as $|\rho| \to \iy$.

Using the relation $S_+^*(\rho) = S_+(-\rho)$ and Condition B5, we derive
$$
 	(I_m - S_+(\rho) S_+^*(\rho)) S_+(\rho) = S_+(\rho) (I_m - S_+^*(\rho) S_+(\rho)),
$$
$$
 	(D^*(-\rho))^{-1} (D(-\rho))^{-1} S_+(\rho) = S_+^*(-\rho) (D^*(\rho))^{-1} (D(\rho))^{-1}.
$$
Taking \eqref{recS-} into account, we obtain
\begin{equation} \label{smeqS-}
 	S_-^*(\rho) = - (D(-\rho))^{-1} S_+(\rho) D(\rho) = - D^*(-\rho) S_+^*(-\rho) (D^*(\rho))^{-1} = S_-(-\rho),
\end{equation}
for real $\rho \ne 0$.

Furthermore,
\begin{multline*}
 	S_-^*(\rho) S_-(\rho) = D^*(-\rho) S_+^*(-\rho) S_+(-\rho) (D^*(-\rho))^{-1} \\ 
 	= D^*(-\rho) \left\{ I_m - (D^*(-\rho))^{-1} (D(-\rho))^{-1} \right \} (D^*(-\rho))^{-1} = I_m - (D(-\rho))^{-1} (D^*(-\rho))^{-1}.
\end{multline*}
Since $A(\rho) := D^*(-\bar \rho)$, we have
\begin{equation*}
S_-^*(\rho) S_-(\rho) = I_m - (A^*(\rho))^{-1} (A(\rho))^{-1}. 
\end{equation*}
Consequently, $\| S_-(\rho) \| < 1$ for real $\rho \ne 0$.

Since $S_-(\rho) \in L_2((-\iy, \iy); \mathbb{C}^{m \times m})$, the Fourier transform $R_-(x)$ also belongs to $L_2((-\iy, \iy); \mathbb{C}^{m \times m})$. 
The relation $R_-^*(x) = R_-(x)$ follows from \eqref{smeqS-}.
The estimates \eqref{estR} follow from Condition B7.

Finally, $(N_k^-)^* = N_k^- \ge 0$ follows from \eqref{Nkconnect} and Condition A$_+$.
\end{proof}

Thus, the data $J_{\pm}$ satisfy Condition A$_{\pm}$. By Theorem~\ref{thm:GLMsol}, the Gelfand-Levitan-Marchenko equations
\eqref{GLM+}, \eqref{GLM-} have unique solutions $K_{\pm}(x, y)$. Then one can construct matrix functions
$F_{\pm}(x, \rho)$ by \eqref{transform} and $Q_{\pm}(x)$ by \eqref{recQ}, satisfying the assertion of Lemma~\ref{lem:Qpm}.
Our next goal is to prove that $Q_+(x) \equiv Q_-(x)$ and $J_{\pm}$ are their scattering data.

\begin{lem} \label{lem:tech}
The following relations hold
\begin{equation} \label{SAD1}
   F_-(x, \rho) S_-(\rho) + F_-(x, -\rho) = F_+(x, \rho) (A(\rho))^{-1}, 
\end{equation}
\begin{equation} \label{SAD2}
   F_+(x, \rho) S_+(\rho) + F_+(x, -\rho) = F_-(x, \rho) (D(\rho))^{-1},	
\end{equation}
\begin{equation} \label{FRN}
  F_+(x, \rho_k) R_k^- = F_-(x, \rho_k) i N_k^-, \quad F_-(x, \rho_k) R_k^+ = F_+(x, \rho_k) i N_k^+, \quad k = \overline{1, N},
\end{equation}
for $x < -a$ and $x > a$, where $a$ is sufficiently large.
\end{lem}

\begin{proof}
Divide the proof into three steps.

{\bf Step 1.} Denote
$$
 	\Phi(x, y) = R_+(x + y) + \int_x^{\iy} K_+(x, t) R_+(t + y) \, dt.
$$
For each fixed $x$, $\Phi(x, y) \in L_2((-\iy, \iy); \mathbb{C}^{m \times m})$.
It follows from \eqref{transform} and \eqref{inverseR}, that
\begin{multline} \label{smeq6}
F_+(x, \rho) S_+(\rho) = \left( \exp(i \rho x) \, I_m + \int_x^{\iy} K_+(x, t) \exp(i \rho t) \, dt \right) \cdot
\int_{-\iy}^{\iy} R_+(\xi) \exp(-i\rho\xi) \, d\xi \\ 
= \int_{-\iy}^{\iy} R_+(x + y) \exp(-i\rho y)\, dy + \int_x^{\iy} K_+(x, t) \int_{-\iy}^{\iy} R_+(t + y) \exp(-i\rho y)\, dy \, dt \\ =
\int_{-\iy}^{\iy} \Phi(x, y) \exp(- i \rho y) \, dy.
\end{multline}
On the other hand, \eqref{GLM+}, \eqref{defM} and \eqref{transform} imply
$$
 	\Phi(x, y) = -K_+(x, y) - \sum_{k = 1}^N \exp(-\tau_k y) F_+(x, \rho_k) N_k^+, \quad x < y.
$$
Therefore
\begin{multline} \label{smeq7}
\int_{-\iy}^{\iy} \Phi(x, y) \exp(- i \rho y) \, dy = \int_{-\iy}^x \Phi(x, y) \exp(-i \rho y) \, dy + 
\exp(- i \rho x) \, I_m - F_+(x, -\rho) \\ - \sum_{k = 1}^N \exp(- i\rho x) \frac{\exp(-\tau_k x)}{\tau_k + i \rho} F_+(x, \rho_k) N_k^+.
\end{multline}
Combining \eqref{smeq6} and \eqref{smeq7}, we obtain
\begin{equation} \label{smeqH}
 	F_+(x, \rho) S_+(\rho) + F_+(x, -\rho) = H_-(x, \rho) (D(\rho))^{-1},
\end{equation}
where
\begin{multline} \label{defH-}
   H_-(x, \rho) = \biggl[ \int_{-\iy}^x \Phi(x, y) \exp(-i\rho y)\, dy + \exp(-i\rho x) \, I_m \\ - 
   \sum_{k = 1}^N  \frac{\exp(-(i \rho + \tau_k) x)}{\tau_k + i \rho} F_+(x, \rho_k) N_k^+\biggr] D(\rho).
\end{multline}

Rewrite \eqref{smeqH} in the form
$$
 	H_-(x, \rho) = (F_+(x, \rho) S_+(\rho) + F_+(x, -\rho)) D(\rho).
$$
Clearly, $H_-(x, \rho)$ is continuous for real $\rho \ne 0$. Moreover, it follows from Condition B6, that
\begin{equation} \label{limH}
 	\lim_{\rho \to 0} \rho H_-(x, \rho) = 0_m
\end{equation}
for each fixed $x$, so $\rho H_-(x, \rho)$ is continuous for $\rho \in \mathbb{R}$. 

In view of \eqref{defH-}, the matrix function $H_-(x, \rho)$ is analytic in the upper half-plane, except for the poles $\rho = \rho_k$.
But
$$
 	\Res_{\rho = \rho_k} H_-(x, \rho) = i F_+(x, \rho_k) N_k^+ D(\rho_k).
$$
According to Codition B2, $\Res\limits_{\rho = \rho_k} (D(\rho))^{-1} = C_k N_k^+$, so
$N_k^+ D(\rho_k) = 0_m$. Hence, $\rho H_-(x, \rho)$ is analytic for $\mbox{Im}\,\rho > 0$ and
continuous for $\mbox{Im} \, \rho \ge 0$, and \eqref{limH} is valid for $\mbox{Im} \, \rho \ge 0$.

Using \eqref{defH-}, we get
\begin{equation} \label{ResH-}
 	\Res_{\rho = \rho_k} H_-(x, \rho) (D(\rho))^{-1} = H_-(x, \rho_k) R_k^+ =  F_+(x, \rho_k) i N_k^+,
\end{equation}
\begin{equation} \label{limHinf}
 	\lim_{|\rho| \to \iy} H_-(x, \rho) \exp(i \rho x) = I_m, \quad \mbox{Im}\,\rho \ge 0.
\end{equation}

Similarly one can show that
\begin{equation} \label{smeqH+}
 	F_-(x, \rho) S_-(\rho) + F_-(x, -\rho) = H_+(x, \rho) (A(\rho))^{-1},
\end{equation}
where the matrix function $H_+(x, \rho)$ is analytic for $\mbox{Im}\,\rho > 0$, 
$\rho H(x, \rho)$ is continuous for $\mbox{Im}\,\rho \ge 0$, and
\begin{equation} \label{ResH+}
 	H_+(x, \rho_k) R_k^- = F_-(x, \rho_k) i N_k^-,
\end{equation}
\begin{equation} \label{limH+}
 	\lim_{|\rho| \to \iy} H_+(x, \rho) \exp(-i\rho x) = I_m, \quad \lim_{\rho \to 0} \rho H_+(x, \rho) = 0_m, \quad \mbox{Im}\,\rho \ge 0.
\end{equation}
Note that for the proof, one needs the relation
$$
  \lim_{\rho \to 0} \rho (S_-(\rho) - I_m) A(\rho) = 0,
$$
which can be easily derived from Condition B6.

{\bf Step 2.}
It follows from \eqref{smeqH}, that
\begin{equation} \label{smeqH-}
 	F_+(x, -\rho) S_+(-\rho) + F_+(x, \rho) = H_-(x, -\rho) (D(-\rho))^{-1}. 
\end{equation}
Substitute $F_+(x, -\rho)$ from \eqref{smeqH} into \eqref{smeqH-}:
$$
 	F_+(x, \rho) (I_m - S_+(\rho) S_+(-\rho)) =  H_-(x, -\rho) (D(-\rho))^{-1} -
 	H_-(x, \rho) (D(\rho))^{-1} S_+(-\rho)
$$
Using Condition B6, \eqref{smeqS-} and the relation $A(\rho) = D^*(-\rho)$, we obtain
\begin{equation} \label{smeqH2}
  F_+(x, \rho) (A(\rho))^{-1} = H_-(x, -\rho) + H_-(x, \rho) S_-(\rho). 
\end{equation}

Put $\bar H_+(x, \rho) := H_+^*(x, -\rho)$. Taking \eqref{barF} and \eqref{smeqS-} into account, 
we obtain from \eqref{smeqH+}:
\begin{equation} \label{smeqH3}
  	(D(\rho))^{-1} \bar H_+(x, \rho) = \bar F_-(x, -\rho) + S_-(\rho) \bar F_-(x, \rho). 
\end{equation}
It follows from \eqref{smeqH2} and \eqref{smeqH3}, that
\begin{multline} \label{defG}
  H_-(x, -\rho) \bar F_-(x, \rho) - H_-(x, \rho) \bar F_-(x, -\rho) \\ = 
  F_+(x, \rho) (A(\rho))^{-1} \bar F_-(x, \rho) - H_-(x, \rho) (D(\rho))^{-1} \bar H_+(x, \rho) =: G(x, \rho).
\end{multline}
For each fixed $x$, the matrix function $G(x, \rho)$ is meromorphic in the upper half-plane with 
the simple poles $\rho_k$. Using \eqref{defG}, we derive
\begin{equation} \label{smRes1}
 	\Res_{\rho = \rho_k} G(x, \rho_k) = F_+(x, \rho_k) R_k^- \bar F_-(x, \rho_k) - H_-(x, \rho_k) R_k^+ \bar H_+(x, \rho_k).
\end{equation}
The relation \eqref{ResH+} yields
\begin{equation} \label{smRes2}
 	R_k^+ \bar H_+(x, \rho_k) = i N_k^- \bar F_-(x, \rho_k).
\end{equation}
Multiplying the relation \eqref{ResH-} by $i (N_k^+)^{-1} (R_k^+)^*$ ($(N_k^+)^{-1}$ wa defined in Lemma~\ref{lem:connect})
and using \eqref{Nkconnect}, we obtain
\begin{equation} \label{smRes3}
 	F_+(x, \rho_k) R_k^- = H_-(x, \rho_k) i N_k^-.
\end{equation}
Substituting \eqref{smRes2} and \eqref{smRes3} into \eqref{smRes1}, we get
$$
   \Res_{\rho = \rho_k} G(x, \rho_k) = 0_m,
$$
so the matrix function $G(x, \rho)$ is analytic for $\mbox{Im}\,\rho > 0$.
According to \eqref{defG}, \eqref{limH}, \eqref{limHinf}, \eqref{limH+}, \eqref{transform} and Conditions B1, B4, B5, the matrix function $\rho G(x, \rho)$
is continuous for $\mbox{Im}\,\rho \ge 0$ and
\begin{equation} \label{limG0}
 	\lim_{\rho \to 0} \rho G(x, \rho) = 0_m, 
\end{equation} 	
\begin{equation} \label{limGinf}
 	\lim_{|\rho| \to \iy} G(x, \rho) = 0_m,  
\end{equation}
where $\mbox{Im}\,\rho \ge 0$.

It follows from \eqref{defG}, that
\begin{equation} \label{-G}
  	G(x, \rho) = -G(x, -\rho)
\end{equation}
for real $\rho \ne 0$.
One can continue the matrix function $G(x, \rho)$ to the lower half-plane by formula \eqref{-G}.	
Then $G(x, \rho)$ is analytic in $\mathbb C \backslash \{ 0 \}$. By virtue of \eqref{limG0}, the point $\rho = 0$ 
is a removable singularity, and $G(x, \rho)$ is entire in $\rho$. The relation \eqref{limGinf} holds for all $\rho$,
consequently, by Liuville's theorem, $G(x, \rho) \equiv 0_m$, i.e.
\begin{equation} \label{symH}
H_-(x, -\rho) \bar F_-(x, \rho) = H_-(x, \rho) \bar F_-(x, -\rho), \quad \rho \in \mathbb{R}.
\end{equation}

{\bf Step 3.} 
Similarly to Step 2, starting from \eqref{smeqH+} and \eqref{smeqH3}, one can derive
\begin{multline*}
F_-(x, -\rho) \bar F_-(x, \rho) - F_-(x, \rho) \bar F_-(x, -\rho) = H_+(x, \rho) (A(\rho))^{-1} \bar F_-(x, \rho) \\ 
- F_-(x, \rho) (D(\rho))^{-1} \bar H_+(x, \rho) =: \tilde G(x, \rho).
\end{multline*}
One can show that $\tilde G(x, \rho) \equiv 0_m$ similarly to $G(x, \rho) \equiv 0_m$. Conseqeuntly,
\begin{equation} \label{symF}
 	F_-(x, -\rho) \bar F_-(x, \rho) = F_-(x, \rho) \bar F_-(x, -\rho), \quad \rho \in \mathbb{R}.
\end{equation}
The relations \eqref{symH} and \eqref{symF} together yield
$$
  	P(x, \rho) = P(x, -\rho), \quad P(x, \rho) := H_-(x, \rho) (F_-(x, \rho))^{-1}, \quad \rho \in \mathbb R, 
$$ 
for $x$ such that $\det F_-(x, \pm \rho) \ne 0$. It follows from the properties of the Jost solution and \eqref{limHinf},
that for $x < -a$ (where $a$ is sufficiently large), $P(x, \rho)$ is analytic for $\mbox{Im}\,\rho > 0$ and continuous for 
$\mbox{Im}\,\rho \ge 0$, $\rho \ne 0$,  
and $\lim\limits_{|\rho| \to \iy} P(x, \rho) = I_m$. We continue the matrix function $P(x, \rho)$ to the half-plane
$\mbox{Im}\,\rho < 0$ by the formula $P(x, \rho) = P(x, -\rho)$. In view of \eqref{limH}, the singularity of $P(x, \rho)$ 
is removable at $\rho = 0$. Thus, the matrix function $P(x, \rho)$ is entire and bounded. By Liouville's theorem,
$P(x, \rho) \equiv I_m$, i.e. $H_-(x, \rho) \equiv F_-(x, \rho)$ for $x < -a$.
Now the relations \eqref{SAD1}, \eqref{SAD2}, \eqref{FRN} follow from \eqref{smeqH2}, \eqref{smeqH}, \eqref{ResH-} and \eqref{smRes3}
for $x < -a$. Symmetrically one can prove them for $x > a$, if $a$ is sufficiently large.

\end{proof}

Let us finish the proof of Theorem~\ref{thm:NSC}. It follows from \eqref{SAD2}, that
$$
 	-F''_-(x, \rho) + Q_+(x) F_-(x, \rho) = \rho^2 F_-(x, \rho)
$$
for $x < -a$. Consequently, $Q_+(x) = Q_-(x)$ for $x < -a$, the potential $Q_+(x)$ satisfies the condition \eqref{xQ},
and the Jost solution $\tilde F_-(x, \rho)$ for $Q_+(x)$ equals $F_-(x, \rho)$ for $x < -a$. Construct for $Q_+(x)$ the matrix
functions
$$
   \tilde A(\rho) = -\frac{1}{2 i \rho} \langle \bar F_-(x, \rho), F_+(x, \rho) \rangle, \quad
   \tilde B(\rho) = \frac{1}{2 i \rho} \langle \bar F_-(x, -\rho), F_+(x, \rho) \rangle, \quad x < - a.
$$
Then for $x < -a$ the relation \eqref{AB} holds:
$$
   F_+(x, \rho) = F_-(x, -\rho) \tilde A(\rho) + F_-(x, \rho) \tilde B(\rho).
$$
Comparing this relation with \eqref{SAD1}, we obtain 
$$ 
	\tilde A(\rho) = A(\rho), \quad \tilde B(\rho) = S_-(\rho) A(\rho).
$$
Using these relations and \eqref{FRN}, we conclude that $J_- = \{ S_-(\rho), \rho_k, N_k^- \}$ are the left scattering data
for the potential $Q_+(x)$. Similarly, using \eqref{SAD2}, we prove that $J_+ = \{ S_+(\rho), \rho_k, N_k^+ \}$ 
are the right scattering data for $Q_+(x)$. Symmetrically one can use the relations of Lemma~\ref{lem:tech}
to prove that $J_-$ and $J_+$ are the scattering data for the potential $Q_-(x)$. By virtue of 
the uniqueness theorem (Corollary~\ref{cor:uniq}), $Q_+(x) \equiv Q_-(x)$ and $J_+$, $J_-$ are the scattering data for this potential. 
Theorem~\ref{thm:NSC} is proved.

\medskip

{\bf 5.2. Reflectionless potentials}. Let us consider the case $S(\rho) \equiv 0_m$.

\begin{thm}
For the data $S_+(\rho) \equiv 0_m$, $\rho_k$ and $N_k^+$, $k = \overline{1, N}$, to be the right scattering data
for a certain potential $Q = Q^*$, satisfying \eqref{xQ}, it is necessary and sufficient to satisfy the following conditions:
$\rho_k = i \tau_k$, $\tau_k > 0$, numbers $\rho_k$ are distinct, $N_k^+ = (N_k^+)^* \ge 0$.
\end{thm}

\begin{proof}
The necessity part follows directly from Theorem~\ref{thm:NSC}. In order to prove the sufficiency part, we
have to check Conditions A$_+$ and B.
It is sufficent to construct the matrix function $U(\rho)$,
meromorphic with the simple poles $\rho_k$ and zeros $- \rho_k$, such that
\begin{equation} \label{ResU}
   \Res_{\rho = \rho_k} U(\rho) = C_k N_k^+, \quad k = \overline{1, N}, \, \det C_k \ne 0,
\end{equation}
\begin{equation} \label{smeqU}
 	U(\rho) = I_m + O(\rho^{-1}), \quad |\rho| \to \iy, \quad U^*(\rho) U(\rho) = I_m, \quad \rho \in \mathbb{R}.
\end{equation}
Following \cite{ZMNP}, let us find such function in the form
$$
 	U(\rho) = \Bigl( I_m + \frac{2 \rho_N}{\rho - \rho_N} P_N \Bigr) \dots \Bigl( I_m + \frac{2 \rho_2}{\rho - \rho_2} P_2 \Bigr) \Bigl( I_m + \frac{2 \rho_1}{\rho - \rho_1} P_1 \Bigr),
$$
where $P_k$ are orthogonal projectors. Note that $U(\rho)$ does not have any other poles and zeros except for $\rho_k$ and $-\rho_k$, 
$k = \overline{1, N}$, respectively, and fulfills \eqref{smeqU}. It remains to choose the projectors $P_k$ to fulfill \eqref{ResU}.
For $k = 1$ we have
$$
  \Res_{\rho = \rho_1} U(\rho) = \Bigl( I_m + \frac{2 \rho_N}{\rho_1 - \rho_N} P_N \Bigr) \dots \Bigl( I_m + \frac{2 \rho_2}{\rho_1 - \rho_2} P_2 \Bigr) 2 \rho_1 P_1 = C_1 N_1^+. 
$$
Consequently, $\mbox{Ker}\, P_1 = \mbox{Ker}\, N_1^+$, so take such $P_1$, that $I_m - P_1$ is an orthogonal projector to $\mbox{Ker}\, N_1^+$.
Further,
$$
  \Res_{\rho = \rho_2} U(\rho) = \Bigl( I_m + \frac{2 \rho_N}{\rho_2 - \rho_N} P_N \Bigr) \dots 2 \rho_2 P_2 \Bigl( I_m + \frac{2 \rho_1}{\rho_2 - \rho_1} P_1 \Bigr)  = C_2 N_2^+. 
$$
We have $\mbox{Ker}\, P_2 = \mbox{Ker}\, N_2^+ (I + \frac{2\rho_1}{\rho_2 - \rho_1} P_1)^{-1}$, so make $I_m - P_2$ to be 
a projector to the space $(I_m + \frac{2 \rho_1}{\rho_2 - \rho_1} P_1) \mbox{Ker}\, N_2^+$ and continue the process for $k = 3, \dots, N$.
Clearly, the matrix function $D(\rho) = (U(\rho))^{-1}$ satisfies Condition B, so the sufficiency part is proved.
\end{proof}

\begin{remark}
In the general case, we can search for the matrix function $D(\rho)$ in the form $D(\rho) =  (U(\rho))^{-1} H(\rho)$,
where $U(\rho)$ is the matrix function, constructed above by the discrete scattering data, and $H(\rho)$
satisfies Condition B with $\det H(\rho) \ne 0$, $\mbox{Im}\,\rho \ge 0$, $\rho \ne 0$ instead of B2 and
$$
 	(H^*(\rho))^{-1} (H(\rho))^{-1} = U(\rho) (I - S_+^*(\rho) S_+(\rho)) U^*(\rho)
$$
instead of B5.  
\end{remark}

\medskip

{\bf 5.3. Riemann problem and application to the matrix KdV equation}
We note that Condition B is related to the following matrix Riemann problem.

\smallskip

{\bf Riemann problem (RP).} {\it Given a continuous matrix function $G(\rho)$, $\rho \in \mathbb{R} \backslash \{ 0 \}$,
distinct numbers $\rho_k$, $k = \overline{1, N}$, $\mbox{Im}\, \rho_k > 0$, and matrices $N_k^+$.
Find a matrix function $D(\rho)$ with the following properties:

\begin{enumerate}
\item $D(\rho)$ is analytical for $\mbox{Im}\,\rho > 0$ and $\rho D(\rho)$ is continuous for $\mbox{Im}\,\rho \ge 0$.

\item $\det D(\rho) = 0$ only for $\rho = \rho_k$, $k = \overline{1, N}$, and $\Res\limits_{\rho = \rho_k} (D(\rho))^{-1} = C_k N_k^+$
for some $C_k \in \mathbb{C}^{m \times m}$, $\det C_k \ne 0$.

\item $D(\rho) = I_m + O(\rho^{-1})$, as $|\rho| \to \iy$.

\item $(D(\rho))^{-1} = O(1)$, as $\rho \to 0$.

\item $(D^*(\rho))^{-1} (D(\rho))^{-1} = G(\rho)$ for real $\rho \ne 0$.
\end{enumerate}
}
\smallskip

There is an extensive literature devoted to the Riemann (or Riemann-Hilbert) problem (see, for example, \cite{GK58, AF03, Gakhov}). 
However, the formulated problem RP differs from the classical case by the singularity at $\rho = 0$.
Here we prove the uniqueness theorem for RP.

\begin{thm}
If RP has a solution, this solution is unique.
\end{thm} 

\begin{proof}
Let, on the contrary, RP have two solutions $D_1(\rho)$ and $D_2(\rho)$.
Consider the matrix function $U(\rho) := (D_2(\rho))^{-1} D_1(\rho)$. Clearly, 
this function is analytical for $\mbox{Im}\,\rho > 0$ and continuous for $\mbox{Im}\,\rho \ge 0$, $\rho \ne 0$.
Moreover,
\begin{equation} \label{symU}
   U^*(\rho) U(\rho) = I_m, \quad \rho \in \mathbb{R} \backslash \{ 0 \},
\end{equation}
$$
   U(\rho) = I_m + O(\rho^{-1}), \quad |\rho| \to \iy, \quad U(\rho) = O(\rho^{-1}), \quad \rho \to 0.
$$
The matrix function $V(\rho) = U(\frac{1}{\rho})$ is analytical in the lower half-plane and continuous for
$\mbox{Im}\, \rho \le 0$, $V(0) = I_m$ and $V(\rho) = O(\rho)$ as $|\rho| \to \iy$.
By virtue of \eqref{symU}, the matrix function $V(\rho)$ is bounded for real $\rho$.
It follows from the Phragmen-Lindel\"of theorem \cite{Titch32}, that $V(\rho) = O(1)$ as $|\rho|\to \iy$ in the lower half-plane.
Hence $U(\rho) = O(1)$ as $\rho \to 0$.

Consider the closed contour $C_{\de, R}$, being the boundary of the region $\{ \rho \colon \mbox{Im}\,\rho \ge 0, \, \de < \rho < R \}$.
Cauchy's integral formula yields
$$
  	U(\rho) - I_m = \frac{1}{2 \pi i}\int\limits_{C_{\de, R}} \frac{U(\xi) - I_m}{\xi - \rho} \, d\xi, \quad \mbox{Im}\,\rho > 0.
$$  
Passing to the limit as $R \to \iy$ and $\de \to 0$ and taking the behavior of $U(\rho)$ at infinity and zero into account,
we obtain
\begin{equation} \label{intU1}
  U(\rho) - I_m = \frac{1}{2 \pi i} \int\limits_{-\iy}^{\iy} \frac{U(\xi) - I_m}{\xi - \rho} \, d\xi, \quad \mbox{Im}\,\rho > 0.
\end{equation}
Here and below when necessary, the integral is considered in the principle value sense.
For $\rho \in \mathbb{R} \backslash \{ 0 \}$, we derive
\begin{equation} \label{intU2}
  U(\rho) - I_m = \frac{1}{\pi i} \int\limits_{-\iy}^{\iy} \frac{U(\xi) - I_m}{\xi - \rho}\, d\xi.
\end{equation}

Note that $\det U(\rho) \ne 0$ for $\mbox{Im}\,\rho \ge 0$, $\rho \ne 0$, and the matrix function $(U(\rho))^{-1}$ has the similar properties
as $U(\rho)$. So we obtain the formula
$$
  (U(\rho))^{-1} - I_m = \frac{1}{\pi i} \int\limits_{-\iy}^{\iy} \frac{(U(\xi))^{-1} - I_m}{\xi - \rho}\, d\xi, \quad \rho \in \mathbb{R} \backslash \{ 0 \}.
$$
Using \eqref{symU}, we get
$$
  U^*(\rho) - I_m = \frac{1}{\pi i} \int\limits_{-\iy}^{\iy} \frac{U^*(\xi) - I_m}{\xi - \rho}\, d\xi, \quad \rho \in \mathbb{R} \backslash \{ 0 \}.
$$
On the other hand, it follows from \eqref{intU2}, that
$$
  U^*(\rho) - I_m = -\frac{1}{\pi i} \int\limits_{-\iy}^{\iy} \frac{U^*(\xi) - I_m}{\xi - \rho} \, d\xi, \quad \rho \in \mathbb{R} \backslash \{ 0 \}.
$$
Consequently, $U^*(\rho) - I_m = 0_m$ for real $\rho \ne 0$. Together with \eqref{intU1} this implies $U(\rho) \equiv I_m$,
$\mbox{Im}\,\rho \ge 0$, $\rho \ne 0$. Thus, the solution of RP is unique.

\end{proof}

Now consider the matrix Korteweg-de Vries equation
$$
   Q_t = 3 Q Q_x + 3 Q_x Q - Q_{xxx}.
$$
It is well known \cite{Gonch01}, that the evolution of the scattering data for it is described by the following relations:
$$
 	S_+(\rho, t) = S_+(\rho, 0) \exp(8 i \rho^3 t), \quad \rho_k(t) = \rho_k(0), \quad N_k^+(t) = N_k^+(0) \exp(8 i \rho_k^3 t), \quad k = \overline{1, N}.
$$
If the matrix function $D(\rho, 0)$ satisfies Condition B for the initial scattering data $J_+(0) = \{ S_+(\rho, 0), \rho_k(0), N_k^+(0) \}$,
then it solves RP for $G(\rho) = I_m - S_+^*(\rho, t) S_+(\rho, t)$, $\rho_k(t)$, $N_k^+(t)$ for every $t$.
Since this solution is unique, it remains to check the estimates \eqref{estR} for $R_{\pm}(x, t)$. The other requirements 
of Conditions A$_+$ and B are trivial.

\medskip

{\bf Acknowledgments}. This work was supported by Grant 1.1436.2014K
of the Russian Ministry of Education and Science and by Grants 14-01-31042 and 15-01-04864
of Russian Foundation for Basic Research.

\medskip

\noindent Natalia Bondarenko \\
Department of Mechanics and Mathematics, Saratov State University, \\
Astrakhanskaya 83, Saratov 410012, Russia, \\
e-mail: {\it BondarenkoNP@info.sgu.ru}

\end{document}